\documentclass[10pt, a4paper]{article}

\usepackage{color,graphicx}
\usepackage{bbm}
\RequirePackage[OT1]{fontenc}
\RequirePackage[utf8]{inputenc}
\usepackage[colorinlistoftodos]{todonotes}
\usepackage{newunicodechar}
\newunicodechar{ϵ}{\varepsilon}
\usepackage{comment}

\RequirePackage{amsthm,amsmath,amsfonts,amssymb,amscd,mathrsfs}
\RequirePackage[colorlinks,citecolor=blue,urlcolor=blue]{hyperref}
\usepackage{enumitem}
\usepackage{mathtools}
\usepackage{authblk}
\usepackage[style=numeric,sorting=nyt]{biblatex}
\usepackage{cancel}

\DeclareUnicodeCharacter{03A9}{\ensuremath{\Omega}}

\usepackage[
    left=1in,
    right=1in,
    top=1in,
    bottom=1in
]{geometry}

\graphicspath{ {./figures/} }

\numberwithin{equation}{section}

\newtheorem{definition}{Definition}[section]
\newtheorem{theorem}[definition]{Theorem}

\newtheorem{lemma}[definition]{Lemma}
\newtheorem{proposition}[definition]{Proposition}

\newtheorem{assumption}[definition]{Assumption}

\newtheorem{remark}[definition]{Remark}

\renewcommand{\P}{\mathbb{P}}

\newcommand{\E}{\mathbb{E}}
\newcommand{\R}{\mathbb{R}}

\newcommand{\N}{\mathbb{N}}

\newcommand{\eps}{\varepsilon}

\renewcommand\labelenumi{(\roman{enumi})}
\renewcommand\theenumi\labelenumi
\newenvironment{enumerate*}%

\makeatletter  
\@ifundefined{@currsizeindex} 
  {\RequirePackage{relsize}\let\dolarger\relsize} 
  {\def\dolarger#1{\larger[#1]}} 

\newcommand*\@@bigtimes[2]{\vphantom{\prod} 
  \vcenter{\hbox{\dolarger{4}$\m@th#1\mkern-2mu\times\mkern-2mu$}}} 
\newcommand*\bigtimes{\mathop{\mathpalette\@@bigtimes\relax}\displaylimits} 
\makeatother

\def\C{\mathbb{C}}\def\E{\mathbb{E}}\def\N{\mathbb{N}}\def\P{\mathbb{P}}\def\R{\mathbb{R}}\def\1{\mathbbm{1}}

\usepackage{titlesec}

\titleformat*{\section}{\Large\bfseries}
\titleformat*{\subsection}{\large\bfseries}
\titleformat*{\subsubsection}{\normalsize\bfseries}

\newlength{\myeqskip}  
\AtBeginDocument{%
    \setlength\abovedisplayskip{\myeqskip}%
    \setlength\belowdisplayskip{\myeqskip}%
    \setlength\abovedisplayshortskip{\myeqskip-7pt}%
    \setlength\belowdisplayshortskip{\myeqskip}}

\allowdisplaybreaks

\usepackage{fancyhdr}
\begin{document}

\thispagestyle{empty}

\pagenumbering{arabic}

\title{Minimax Operator Learning for Holomorphic Maps between Hilbert Spaces}

\author{
    Ly Duc Hoang\footnote{
        Department of Mathematics, Humboldt University of Berlin, Germany.
        Email: {ly.duc.hoang@hu-berlin.de}
    }~,
    Sven Wang\footnote{
        Department of Mathematics, \'Ecole Polytechnique F\'ed\'erale de Lausanne, Switzerland. Email:
        {sven.wang@epfl.ch}
    }~, 
    Jakob Zech\footnote{
        Institute for Mathematics, Heidelberg University, Germany.
        Email: {jakob.zech@uni-heidelberg.de}
    }
}

\maketitle

\begin{abstract}
  We consider the problem of learning nonlinear operators $G_0: \mathcal X\to \mathcal Y$ from finitely many noisy observations. While finite-regularity classes such as Lipschitz- or $C^k$-operators suffer from the \textit{curse of sample complexity} (i.e.~algebraic sample complexity is impossible), algebraic convergence rates can be obtained for holomorphic operators. In this paper we study information-theoretic minimax rates for estimating holomorphic operators. We provide algebraic upper and lower bounds of the metric entropy of holomorphy classes (matching up to an arbitrarily small loss in the exponent) as well as statistical minimax lower bounds in both the random-design and a Gaussian white noise model. For the mean squared error, our novel lower bounds match our upper bounds (up to an arbitrarily small loss in the exponent) and coincide with the convergence rates obtained previously for related holomorphy classes \parencite{reinhardt2026statistical}. Finally, we apply the theory to the parameter-to-solution map and the parameter-to-quantity-of-interest map arising from the elliptic Darcy flow equation.

    \medskip
    \noindent\textbf{Keywords:}{
        operator learning, 
        holomorphic operators, 
        minimax estimation, 
        metric entropy, 
        polynomial chaos expansions,
        parametric partial differential equations
    }
\end{abstract}

\section{Introduction}\label{section: Introduction}
Learning nonlinear maps between function spaces is a central problem in modern scientific computing, uncertainty quantification and operator learning \parencite{kovachki2023neural, li2020fourier, bhattacharya2021model, lu2021learning, nelsen2024operator, kovachki2024operator}. A canonical example is the parameter-to-solution map of a partial differential equation (PDE), mapping a spatially varying coefficient, a forcing term or a domain parametrization to the corresponding PDE solution. From a statistical viewpoint, the central objective is to find methods and quantify the performance in recovering such an a priori unknown operator from finitely many noisy observations. 

Let $\mathcal X$ and $\mathcal Y$ be separable Hilbert spaces and let 
\begin{align*}
    G_0 : \mathcal X \to \mathcal Y
\end{align*}
be some unknown nonlinear operator. A first natural observation model consists of independent input-output pairs $(X_i, Y_i)_{i=1}^n$, where
\begin{equation}\label{eq: random design}
    X_i \overset{\text{i.i.d.}}{\sim} \gamma,\qquad Y_i = G_0(X_i) + \sigma \varepsilon_i, \qquad i = 1, \dots, n,
\end{equation}
$\gamma$ is some probability measure on $\mathcal X$, $\sigma>0$ is some scalar and $(\varepsilon_i)_{i=1}^n$ are independent Gaussian white noise processes indexed by $\mathcal Y$. For $\mathcal X = \R^d$ and $\mathcal Y = \R$, this reduces to the classical random-design nonparametric regression problem. (Statistical) Operator Learning may therefore be viewed as an infinite-dimensional analogue of classical nonparametric regression~\parencite{reinhardt2026statistical}.

A fundamental difficulty here is that both the input and the output function spaces might be infinite-dimensional. Classical finite-dimensional, finite-regularity assumptions in general lead to \textit{the curse of dimensionality}, in the sense that the corresponding deterministic approximation rates and statistical estimation rates are dimension-dependent and deteriorate exponentially with the dimension $d$. If the input and output are infinite-dimensional, even with anisotropic dependence in the input coordinates and stronger regularity assumptions in the output space, broad classes of Lipschitz, or finite-regularity operators suffer from \textit{the curse of parametric complexity} in approximation and the \textit{curse of sample complexity} in statistical estimation \parencite{lanthaler_parametric_2025, adcock2025towards}, leading to rates which decay only logarithmically in the sample size $n$. This raises the following natural question:

\begin{quote}
    Which structural assumptions on an operator are sufficiently strong to \textit{break the curse of sample/parametric complexity}, while being large enough to contain relevant parameter-to-solution operators arising in the context of PDEs?
\end{quote}

One possible such structure is to consider maps which depend analytically/holomorphically on their input parameters. Several works in approximation theory and uncertainty quantification have shown that the parameter-to-solution map of many parametric PDEs admits holomorphic extensions with respect to their anisotropically weighted input coordinates \parencite{cohen_convergence_2010, cohen_analytic_2011, cohen2015approximation, cohen2018shape, harbrecht2016analysis, henriquez2021shape, hiptmair2018large}. These works have in common that the holomorphic dependence implies the sparsity of the corresponding polynomial chaos expansions of the holomorphic map and leads to algebraic approximation rates despite both the input and output being infinite-dimensional.

The goal of this paper is to establish analogous statistical results as well as to provide minimax lower bounds, complementing the upper bounds derived in \parencite{reinhardt2026statistical}. 

\subsection{Main Results}
\paragraph{A Gaussian white noise model.} 
In addition to the random-design model, we consider a Gaussian white noise experiment. White noise models play a central role in classical nonparametric statistics, because they are often mathematically and functional-analytically convenient "continuum limit", see for example \parencite[Section 6.1.1]{gine_mathematical_2021} or \parencite{brown1996asymptotic, reiss_asymptotic_2008} for its relation to the random design model. Writing 
\begin{align*}
    H = L^2(\mathcal X,\mathcal Y, \gamma)
\end{align*}
for the space of square-integrable, Bochner-measurable, $\mathcal Y$-valued operators, we generalize this observation model to general separable Hilbert spaces by observing the Gaussian process
\begin{align*}
    \mathbb Y_{G_0}^{(n)}(G) = \langle G_0, G\rangle_H + \frac{\sigma}{\sqrt{n}}\mathbb W_H(G), \qquad G\in H,
\end{align*}
where $\mathbb W_H$ denotes a Gaussian white noise process on $H$. We refer to this as the \emph{operator white noise model} or if $H$ is just an arbitrary, separable Hilbert space, by the \emph{generalised white noise model}. The Hilbert-space structure makes likelihood calculations and projection methods mathematically tractable. For a more precise mathematical formulation, we refer to Section \ref{section: lower bounds}.

\paragraph{Metric entropy of classes of holomorphic operators} Following the classical notion of metric entropy due to \textcite{kolmogorov1959entropy}, we will first quantify the \textit{size} of the class of holomorphic operators through its metric entropy \parencite{kolmogorov1959entropy} with respect to the $L^2(\mathcal X, \mathcal Y, \gamma)$-norm and with respect to the supremum norm. We show that, despite infinite-dimensional inputs and outputs, the metric entropy only scales algebraically with $\varepsilon\to 0$. For the $L^2$-case, the effective complexity is determined by the parameter 
\begin{align*}
    \kappa_2 = \min\{s-1/2, t\}.
\end{align*}
Up to an arbitrarily small loss in the exponent, the metric entropy then scales like $\varepsilon^{-1/\kappa_2}$.
Analogously, for the supremum loss, the effective complexity is determined by the parameter 
\begin{align*}
    \kappa_\infty = \min\{s-1,t\}
\end{align*}
and the metric entropy scales like $\varepsilon^{-1/\kappa_\infty}$.

The same combinations of input and output regularity parameters $s$ and $t$ appear in the exponent of the $L^2$- and $L^\infty$- approximation rates for neural and spectral operator surrogates \parencite{herrmann_neural_2024}. More details can be found in Section \ref{section: metric entropy}.

\paragraph{Minimax Lower Bounds for Operator Learning of Holomorphic Maps}
The class of holomorphic operators allows for sparse approximation in the polynomial chaos expansion \parencite{cohen_convergence_2010, cohen_analytic_2011,chkifa_breaking_2015}. The summability of the coefficients yields coefficient bounds and algebraic $N$-term approximation rates, which are used in the proof of the upper bounds of the metric entropy and to control the \textit{bias} of the statistical projection estimators based on the polynomial chaos expansion. Balancing this against the stochastic error arising from the observation noise yields the upper bound for the statistical estimation rates. 

The statistical lower bounds rely on information-theoretic arguments (see \textcite{yang1999information} and also \parencite{tsybakov_introduction_2008,gine_mathematical_2021}).  In both observation models, the Kullback-Leibler divergence between the two observation laws corresponding to two operators $G, G'\in L^2(\mathcal X, \mathcal Y, \gamma)$ is directly controlled by 
\begin{align*}
    \mathrm{KL}(\P_{G}, \P_{G'}) = \frac{n}{2\sigma^2}\|G-G'\|^2_{L^2(\mathcal X, \mathcal Y, \gamma)}.
\end{align*}
We study classes of holomorphic operators $\mathcal H^{s,t}$ between Hilbert spaces, where $s>1$ describes the polynomial decay in the input coordinates, while $t>0$ quantifies the decay
in the output space. The precise definition of $\mathcal H^{s,t}$ is given in Section~\ref{section: holomorphy class}.
Following the standard methods for statistical lower bounds, we combine the $L^2$-metric entropy bounds with the construction of suitable packings in $\mathcal H^{s,t}$ with respect to the $L^2$- or $L^\infty$-metric.

This leads to the main statistical result of this paper regarding the minimax lower bounds in both the random design model and the operator white noise model. Note that corresponding upper bounds in the random design model have recently been proven in \parencite{reinhardt2026statistical}. 
We denote by 
\begin{equation*}
    \mathcal H_{\mathrm{lin}}^{s,t}\coloneqq{G\in \mathcal H^{s,t}: G \text{ is affine}}
\end{equation*}
the subset of (affine) linear operators in $\mathcal H^{s,t}$. It turns out that the lower bounds already hold over this smaller class. In particular, the resulting rates match the algebraic $L^2$- minimax lower bounds up to an arbitrarily small loss in the exponent.
\begin{theorem}[Informal: Minimax Operator Learning upper and lower bounds]\label{theorem: main minimax}
    Let $s>1$, $t>0$ and $\sigma >0$, and fix $\tau>0$. Set $\kappa_2 \coloneqq \min\{s-1/2,t\}$. In both the operator white noise model \eqref{eq: white noise} and the random-design white noise model \eqref{eq: random design}, the squared $L^2$-minimax risk over $\mathcal H^{s,t}$ satisfies
    \begin{align}\label{eq: L2 minimax rate}
        n^{- \frac{2\kappa_2}{2\kappa_2 + 1} - \tau}\lesssim \inf_{\widehat G_n}
\sup_{G_0 \in \mathcal H_{\mathrm{lin}}^{s,t}}
\mathbb E_{G_0}[|\widehat G_n-G_0|^2_{L^2(\gamma,\mathcal Y)}]\leq\inf_{\widehat G_n} \sup_{G_0 \in \mathcal H^{s,t}} \mathbb E_{G_0}[\|\widehat G_n - G_0\|^2_{L^2(\gamma, \mathcal Y)}] \lesssim   
            n^{- \frac{2\kappa_2}{2\kappa_2 + 1} + \tau}.
    \end{align}
    The implicit constants are independent of $n$. Here and below, the infimum is taken over all measurable maps $\widehat G_n$ with respect to the data. 
    
    For the $L^\infty$-loss, we obtain the following lower and upper bounds on the minimax risk. Setting $\kappa_\infty \coloneqq\min\{s-1,t\}$, we have 
    \begin{align}\label{eq: Linfty minimax rate}
        n^{-\frac{\kappa_{\infty}}{2\kappa_2 + 1} - \tau}\lesssim \inf_{\widehat G_n}
\sup_{G_0 \in \mathcal H_{\mathrm{lin}}^{s,t}}
\mathbb E_{G_0}[|\widehat G_n-G_0|_{\infty, \mathcal Y}]\leq \inf_{\widehat G_n}
        \sup_{G_0 \in \mathcal H^{s,t}}
        \mathbb E_{G_0}
        \left[
            \|\widehat G_n-G_0\|_{\infty, \mathcal Y}
        \right]\lesssim n^{- \frac{\kappa_{\infty}}{2\kappa_{\infty} +3} + \tau}.
    \end{align}
  \end{theorem}
For more details, consult Section \ref{section: lower bounds} and \ref{section: upper bounds}.

As an illustrative example, we consider the parameter-to-solution map of the elliptic Darcy flow equation, see Section \ref{section: examples} for details. It is well-known that the solution depends holomorphically on the conductivity coefficient \parencite{cohen_convergence_2010, cohen_analytic_2011, herrmann_neural_2024}. We determine the corresponding regularity parameter $s>1$ for the input, and the output regularity parameter $t>0$, which determine the $L^2$ and $L^\infty$-rates. Combined with Theorem \ref{theorem: main minimax}, this shows that the $L^2$-rate is minimax optimal over the associated class of holomorphic operators $\mathcal H^{s,t}$. Furthermore, we prove minimax lower bounds for \textit{quantities of interest} (QoIs), specifically for real-valued, linear functionals of the solution.
Note that the algebraic rates obtained in the present infinite-dimensional setting contrast classical results for finite-dimensional nonparametric regression of holomorphic functions \cite{ibragimov_estimation_1998}.

\subsection{Related Literature}

\paragraph{Approximation Theory for Operator Learning and Parametric Partial Differential Equations}
Recent successful applications of deep learning methods in the sciences have motivated the development of neural network architectures which are able to approximate solutions of partial differential equations (PDEs). Popular examples include for example the Deep Ritz method \parencite{yu2018deep}, the Deep Galerkin method \parencite{sirignano2018dgm} and physics-informed neural networks \parencite{raissi2019physics}. While these methods approximate individual solutions of a PDE, \textit{operator learning} seeks to approximate maps between infinite-dimensional spaces, with a particular focus on parameter-to-solution maps of parametric PDEs. Neural-network based surrogates are commonly referred to as \textit{neural operators}\parencite{kovachki2023neural}, with prominent examples including Fourier neural operators (FNOs) \parencite{li2020fourier}, Graph Neural Operators \parencite{li2020multipole}, DeepONet \parencite{lu2021learning} and PCANet \parencite{bhattacharya2021model, hesthaven2018non}. Universal approximation results have been established for several such neural operator architectures \parencite{deng2021convergence, lanthaler2022error, kovachki2021universal, kovachki2023neural}. 

For broad classes of Lipschitz or finitely Fr\'echet differentiable operators, quantitative approximation rates suffer from the \textit{curse of parametric complexity} \parencite{lanthaler2023operator, lanthaler_parametric_2025}. Finding classes of operators for which this curse can be circumvented, remains a central challenge in Operator Learning. One possible solution to this problem arises from the approximation theory of parametric PDEs. For broad classes of operator equations, the parameter-to-solution maps admit holomorphic extensions with respect to the countably many expansion coefficients of the parameter. Examples include linear and nonlinear elliptic equations \parencite{cohen_convergence_2010, cohen_analytic_2011, chkifa_breaking_2015}, parabolic equations \parencite{hoang2013sparse}, and PDEs on parametrized domains \parencite{jerez2017electromagnetic, cohen2018shape, hiptmair2018large, henriquez2021shape}. The holomorphic dependence of the solution on the parameters implies the sparsity of the corresponding generalized polynomial chaos expansions and dimension-independent $N$-term approximation rates. Building on this approximation theory, \textcite{schwab_deep_2019} studied neural network approximation of holomorphic maps and \textcite{herrmann_neural_2024} established a unified framework for neural and spectral operator surrogates with algebraic expression-rate bounds with respect to the $L^2$ and $L^\infty$- norm. Related approximation results on Gaussian spaces are based on Wiener--Hermite polynomials, with holomorphic extensions to complex strips \parencite{dung_analyticity_2023, marcati2026sparsity, marcati2026neural}.

\paragraph{Statistical Learning Theory in Operator Learning}

In the \textit{noise-free} setting, several works have studied the sample complexity of recovering operators from finitely many exact observations. Going back to earlier work by \textcite{mhaskar1997neural} on the nonlinear n-width of Lipschitz operators with respect to $L^2$, these results have been generalized to general classes of Lipschitz or finite-regularity operators,\parencite{kovachki_data_2024, adcock2024sample}. Therein, the authors show the prevalence of a \textit{curse of sample complexity} for these classes, which states that the reconstruction of such operators with algebraic sample complexity is in general impossible. 

The statistical learning theory for operator learning with noisy observations has received less attention. For the estimation of linear operator from noisy samples, statistical errors have been studied in \parencite{de2023convergence} and lower bounds were studied in \parencite{chagny2025adaptive}. For nonparametric regression of Lipschitz operators from noisy data, \textcite{liu2024deep} were able to prove generalization bounds for neural-network based estimators. For Banach space valued holomorphic operators and noisy observations,  \textcite{adcock_optimal_2024, adcock2025near} derive upper and lower bounds for the approximation in $L^2$.

Results on minimax lower bounds for the nonparametric estimation of nonlinear operators have been sparser.  \textcite{adcock2025towards} derived statistical lower bounds for Lipschitz operators in fixed- and random-design observation model, which were however only tight for exponentially decaying eigenvalues of the design measure. In particular, they show that algebraic rates cannot be achieved for such classes of operators. By contrast, \textcite{reinhardt2026statistical} established algebraic, statistical upper bounds for empirical risk minimizers over classes of so-called \textit{FrameNets} if the operator is holomorphic. 

A recent survey by \textcite{brugiapaglia2026short} highlighted the open question of sharp lower bounds for operator learning under statistical observation noise and ask whether the rates by \textcite{reinhardt2026statistical} are minimax-optimal. Our results address this open question both in the random-design white noise model, with our lower bounds matching the upper bounds up to an arbitrarily small loss in the exponent. Furthermore, we establish the same minimax rates in the generalized Gaussian white noise model, which are attained by certain projection estimators based on the generalized polynomial chaos expansion.

\subsection{Notation}

Throughout the text, set $\N = \{1,2,3,  \dots\}$ and $\N_0 = \N \cup \{0\}$ and let $\R$ and $\C$ denote the real numbers and the complex numbers respectively. For a family of parametrized probability distributions $\{\mathbb P_\theta: \theta \in \Theta\}$ indexed by some parameter set $\Theta$, we denote the expectation under $\mathbb P_\theta$ by $\mathbb E_\theta$. For two real sequences $(a_n)_{n\in \N}$ and $(b_n)_{n\in \N}$, we write $a_n \lesssim b_n$ (respectively $a_n \gtrsim b_n$), if there exists a constant $C>0$ independent of $n$, such that $a_n \leq Cb_n$ (respectively $a_n \geq Cb_n$) for all $n\in \N$. If both $a_n \lesssim b_n$ and $a_n \gtrsim b_n$, we write $a_n \simeq b_n$. 

For a pseudo-metric space $(S,d)$ and $T\subset S$ and $\varepsilon>0$, we denote the $\varepsilon$-covering number of $T$ by
\begin{equation}\label{eq:covering}
  N(\varepsilon,T,d):=\inf\Big\{n\in\N\,:\,T\subseteq \bigcup_{i=1}^n \overline B_{\varepsilon}(x_i)\Big\},
\end{equation}
where $\overline B_\varepsilon(x_i):=\{x\in S\,:\,d(x_i,x)\leq\varepsilon\}$ and $x_i\in S$ for all $i$.
Similarly, $M(\varepsilon,T,d)$ denotes the $\varepsilon$-packing number, i.e. the maximal cardinality of a subset of $T$ whose distinct elements have pairwise distance greater than $\varepsilon$.
For any $\varepsilon>0$, the packing and covering numbers satisfy the elementary inequalities
\begin{equation*}
    M(2\varepsilon, T, d)\leq N(\varepsilon, T, d) \leq M(\varepsilon, T, d).
\end{equation*}

The input and output space are denoted by $\mathcal X$ and $\mathcal Y$ respectively, and are assumed to be separable Hilbert spaces with inner products $\langle\cdot,\cdot\rangle_{\mathcal X}$ and $\langle\cdot,\cdot\rangle_{\mathcal Y}$. For $1\leq p < \infty$ and $\gamma$ some probability measure on $\mathcal X$, $L^p(\mathcal X, \mathcal Y, \gamma)$ denotes the Bochner-space of $p$-integrable maps
\begin{align*}
    L^p(\mathcal X, \mathcal Y, \gamma) = \left\{G: \mathcal X\to \mathcal Y \text{ Bochner-measurable}\,:\, \|G\|_{L^p(\mathcal X, \mathcal Y, \gamma)}\coloneqq\left(\int_{\mathcal X} \|G(x)\|_{\mathcal Y}^p \;d\gamma(x)\right)^{1/p}< \infty\right\}.
\end{align*}
For a subset $D\subset \mathcal X$ and $p= \infty$, we define 
\begin{align*}
    L^\infty(D, \mathcal Y) = \left\{G:  D\to \mathcal Y: \|G\|_{\infty, D}\coloneqq \sup_{x\in D}\|G(x)\|_{\mathcal Y} < \infty\right\}.
\end{align*}
If $D$ and $\mathcal X$ are clear from the context, we shorten the notation of the corresponding norms to $\|\cdot\|_{L^p(\gamma, \mathcal Y)}$ and $\|\cdot\|_{\infty, \mathcal Y}$. Since we will consider classes of functionals as well, we write $\|\cdot\|_{L^p(\gamma)}$ and $\|\cdot\|_{\infty}$ if $\mathcal Y = \R$.
The space of real-valued, $p$-summable sequences is denoted by $\ell^p(\N)$. Lastly, the complexification of a real Banach space $(B, \|\cdot\|_B)$ is denoted by $B_{\C}$ \parencite{munoz1999complexifications}.

\section{A Class of Holomorphic Maps Between Hilbert Spaces}\label{section: holomorphy class}

\subsection{Definition of $\mathcal H^{s,t}$}\label{subsection: definition of holomorphy class}
For $\varrho >1$, let 
\begin{align*}
    \mathcal E_\varrho = \left\{\frac{z+ z^{-1}}{2}: z \in \mathbb C,\ 1\leq |z|\leq \varrho\right\}\subset \mathbb C
\end{align*}
denote the (closed) Bernstein ellipse with parameter $\varrho$.
For a sequence $\boldsymbol\varrho \in (1, \infty)^{\N}$, we write
\begin{align*}
    \mathcal E^\infty_{\boldsymbol\varrho} = \bigtimes_{j=1}^\infty \mathcal E_{\varrho_j}
\end{align*}
for the corresponding tensorized Bernstein poly-ellipse. 

Fix Riesz bases $\Psi = (\psi_j)_{j\in \N}$ of $\mathcal X$ and $\Phi = (\phi_j)_{j\in \N}$ of $\mathcal Y$, with the dual (biorthogonal) bases $\widetilde \Psi = (\widetilde \psi_j)_{j\in \N}$ and $\widetilde \Phi = (\widetilde \phi_j)_{j\in \N}$, respectively. 

Throughout this paper, let $\mathbf{w} = (w_j)_{j\in \N}$ satisfy $w_j \simeq 1/j$. For $s>1$, define
\begin{align}\label{eq: input expansion}
    \sigma_s: Q \to \mathcal X, \quad \sigma_s(c) = \sum_{j=1}^\infty c_j j^{-s}\psi_j, \qquad Q= [-1,1]^{\N}.
\end{align}
This parametrizes the input through a weighted expansion with respect to the Riesz basis $\Psi = (\psi_j)_{j\in \N}$ and by controlling the decay of the corresponding coefficients. 

For all $s, t\geq 0$, define the scales of Hilbert spaces 
\begin{align}\label{def: scale spaces}
    \mathcal X^s &\coloneqq \{x\in \mathcal X: \lVert x\rVert_{\mathcal X^s}^2\coloneqq \sum_{j=1}^\infty |\langle x, \widetilde \psi_j\rangle_{\mathcal X}|^2 j^{2s} < \infty\}\quad \text{and}\nonumber\\
    \mathcal Y^t &\coloneqq \{y\in \mathcal Y: \lVert y\rVert_{\mathcal Y^t}^2\coloneqq \sum_{j=1}^\infty |\langle y, \widetilde \phi_j\rangle_{\mathcal Y}|^2 j^{2t} < \infty\},
\end{align}
which are separable Hilbert spaces. We denote their complexifications by $\mathcal X^s_{\C}$ and $\mathcal Y^{t}_{\C}$ respectively. We also define 
\begin{align}\label{eq: abs summable}
    \mathcal X_{\ell^1} = \{x\in \mathcal X: \|x\|_{\mathcal X_{\ell^1}} \coloneqq \sum_{j=1}^\infty |\langle x, \widetilde \psi_j\rangle| < \infty\}.
\end{align}
and denote its complexification by $\mathcal X_{\ell^1, \C}$.

By definition, $\mathcal X^s$ and $\mathcal Y^t$ are isomorphic to the weighted sequence spaces
\begin{align*}
    \ell^2_s(\N) &= \left\{c\in \ell^2(\N): \|c\|_{\ell^2_s}^2 \coloneqq\sum_{j=1}^\infty |c_j|^2 j^{2s}< \infty\right\}\quad \text{and}\\
    \ell^2_t(\N) &= \left\{c\in \ell^2(\N): \|c\|_{\ell^2_t}^2 \coloneqq\sum_{j=1}^\infty |c_j|^2 j^{2t}< \infty\right\}.
\end{align*}

For any $r\geq 0$ and $1\leq q < \infty$, define
\begin{align*}
    \ell^q_{-r}(\N) &= \{c = (c_j)_{j\in \N}\in \R^{\N}: \|c\|_{\ell^q_{-r}}^q \coloneqq\sum_{j=1}^\infty |c_j|^q j^{-rq}< \infty\}.
\end{align*}
and we denote its complexification by $\ell^q_{-r}(\N, \C)$ and abbreviate it to $\ell^q_{-r}$, if it is clear from the context.

For $s>1$, 
\begin{align*}
    \sup_{c\in Q}\|c\|_{\ell^1_{-s}} \leq \sum_{j=1}^\infty j^{-s}<\infty,
\end{align*}
and therefore $Q \subset \ell^1_{-s}(\N)$.

To allow for parameterizations $\sigma_s$ beyond $Q$, we extend it to the weighted sequence space $\ell^1_{-s}(\N)$. The proof of the following lemma can be found in Appendix~\ref{appendix: proof of weighted seq space}.

\begin{lemma}\label{lemma: weighted seq space}
    Let $s>1$. The map $\widetilde \sigma_s: \ell^1_{-s}(\N)\to \mathcal X_{\ell^1}$ with $\widetilde\sigma_s(c) \coloneqq \sum_{j=1}^\infty c_j j^{-s} \psi_j$, is a linear bijection, and satisfies 
    \begin{align*}
        \|\widetilde\sigma_s(c)\|_{\mathcal X_{\ell^1}}=\|c\|_{\ell^1_{-s}},\qquad c\in\ell^1_{-s}(\N).
    \end{align*}
    Moreover, 
    \begin{align*}
        \|\widetilde\sigma_s(c)\|_{\mathcal X}\leq D_\Psi\|c\|_{\ell^1_{-s}}
    \end{align*}
    and $\widetilde \sigma_s$ is a continuous extension of $\sigma_s: Q \to \mathcal X$. Analogous statements hold for the corresponding complexified spaces.
\end{lemma}
By abuse of notation, we again denote this extension by $\sigma_s$. 
We consider two natural classes of holomorphic maps. The first class consists of \textit{holomorphic operators}, following \parencite[Assumption 2]{reinhardt2026statistical}, for which both the input and output spaces are infinite-dimensional Hilbert spaces. The second class consists of scalar-valued \textit{holomorphic functionals} for which only the input space is infinite-dimensional. 

Consider two complex Hilbert spaces $\mathcal X_{\C}$, $\mathcal Y_{\C}$, and an open set $\mathcal O_{\C}\subset \mathcal X_{\C}$. We say that a map $G:\mathcal O_{\C}\to \mathcal Y_{\C}$ is holomorphic, if it is complex Fr\'echet differentiable; 
see, e.g., \parencite[Theorem 13.16]{mujica1986complex} for a discussion of equivalent definitions.

The next definition introduces the set of holomorphic operators that we investigate in the following. It is similar, but slightly less restrictive (i.e.\ larger) than the classes considered in \cite{zech_sparse-grid_2018,herrmann_neural_2024,reinhardt2026statistical}.
In particular, it is essentially a version of the ``$(b,\eps)$-holomorphy'' assumption \cite{chkifa_breaking_2015}. Specifically,
contrary to \cite{herrmann_neural_2024,reinhardt2026statistical} we assume boundedness not on the set $\mathcal{O}_{\C}$, but only on the smaller set $\mathcal{D}_{\C}$ occurring in the following definition:

\begin{definition}[Class of holomorphic operators]\label{def: holomorphic operator}
  Let $s>1$, $t>0$ and $L>0$. Let $\mathcal X_{\C}$, $\mathcal Y_{\C}$ be complex Hilbert spaces and $Q \coloneqq [-1,1]^{\N}$. 
  Fix a Riesz basis $\Psi = (\psi_j)_{j\in \N}\subset \mathcal X$ and
  let $\sigma_s$ be as in \eqref{eq: input expansion}.
  Let $\sigma_s(Q)\subset \mathcal D_{\C}\subset \mathcal O_{\C}\subset \mathcal X_{\C}$ be nested subsets of $\mathcal X_\C$ containing $\sigma_s(Q)$, such that $\mathcal O_{\C}$ is open in $\mathcal X_\C$, and $\mathcal D_{\C} \subset (\mathcal X_{\ell^1})_{\C}$ is open and bounded in the topology of $(\mathcal X_{\ell^1})_{\C}$. Then we denote by $\mathcal H^{s,t}(\mathcal O_{\C}, \mathcal D_{\C}, \mathcal Y_{\C}, L)$ 
    the set of all operators $G: \sigma_s(Q) \to \mathcal Y$ with extensions to $G: \mathcal O_{\C}\to \mathcal Y_{\C}$ such that 
    \begin{enumerate}
        \item $G$ is holomorphic on $\mathcal O_\C$ and
        \item $\sup_{z\in \mathcal D_{\C}}\|G(z)\|_{\mathcal Y_{\C}^t} \leq  L<\infty$.\label{item: uniform boundedness}
    \end{enumerate}
    If $\mathcal O_{\C}$, $\mathcal D_{\C}$ and $\mathcal X_{\C}$ are fixed and clear from the context, we denote this set by $\mathcal H^{s,t}(L; \mathcal Y_{\C})$.
\end{definition}

\begin{remark}
    The proofs in \cite{herrmann_neural_2024,reinhardt2026statistical} only use boundedness on $\mathcal{D}_{\C}$ as required in Definition
    \ref{def: holomorphic operator}, rather than on
    $\mathcal{O}_{\C}$ as assumed \cite{herrmann_neural_2024,reinhardt2026statistical}. Hence
    the upper bounds derived in these papers carry over straightforward to this weaker assumption. In the present paper, where we wish to derive lower bounds matching the upper bounds, it is crucial to work with a possibly minimal assumption, which explains the introduction of $\mathcal{D}_{\C}$. Showing optimal rates for the class assuming boundedness on $\mathcal{O}_{\C}$ constitutes an open question.
\end{remark}
  
If $\mathcal Y_{\C}$ is finite-dimensional, all norms are equivalent. In particular, the $\mathcal Y_{\C}^t$-norm is equivalent to the $\mathcal Y_{\C}$-norm. Therefore, we can replace the uniform boundedness condition in Definition \ref{def: holomorphic operator} simply by the uniform boundedness with respect to the $\mathcal Y_{\C}$-norm, formally corresponding to the case $t=0$.
\begin{definition}[Holomorphic functionals]\label{definition: holomorphic functional}
    Let $s,L,\mathcal X_\C$, $Q$, $\sigma_s$, $\mathcal O_{\C}$ and $\mathcal D_{\C}$ be as in Definition \ref{def: holomorphic operator}. We denote the set of  functionals $g: \sigma_s(Q)\to \R$ with holomorphic extensions $g: \mathcal O_\C\to \C$ such that $\sup_{z\in \mathcal D_{\C}}|g(z)| \leq  L<\infty$ by $\mathcal H^s(\mathcal O_{\C}, \mathcal D_{\C}, \C, L)$. If $\mathcal O_{\C}$ and $\mathcal D_{\C}$ are clear from the context, we write $\mathcal H^{s}(L; \C)$.
\end{definition}

\begin{remark}\label{rem: compactness}
    By Lemma \ref{lemma: weighted seq space}, the map $\sigma_s: \ell^1_{-s}(\N)_{\C}\to \mathcal X_{\ell^1, \C}$ is continuous and thus $\sigma_s^{-1}(\mathcal D_{\C})$ is an open and bounded subset of $\ell^1_{-s}(\N)_{\C}$. Furthermore, $Q$ is compact in the topology of $\ell^1_{-s}$. 
    Since $Q\subset \sigma_s^{-1}(\mathcal D_{\C})$, it holds that 
    \begin{align*}
        r_*\coloneqq \operatorname{dist}
    (Q, \partial \sigma_s^{-1}(\mathcal D_{\C}))>0
    \end{align*}
    with the distance measured with respect to $\ell^1_{-s}$.
    As a consequence, for every $r\in (0, r_*)$, 
    \begin{align*}
        \{z\in\ell^1_{-s}(\N)_{\C}:
    \operatorname{dist}(z,Q)<r\}
    \subset\sigma_s^{-1}(\mathcal D_{\C})\subset \sigma_s^{-1}(\mathcal O_{\C}).
    \end{align*}
    To consider analytic extension domains in the sequence space $\C^{\N}$, one needs to ensure that the corresponding tensorized Bernstein ellipses $\mathcal E_\varrho^\infty$ do not become too large. 
    For every $z\in \mathcal E_{\boldsymbol{\varrho}}^\infty$, one can choose a $c\in Q$ such that 
    \begin{align*}
        |z_j-c_j|\leq \varrho_j -1, \qquad j\in \N.
    \end{align*}
    This follows for example from \parencite[Lemma B.2.4.]{zech_sparse-grid_2018}. Thus, for every
    $z\in \mathcal E_{\boldsymbol{\varrho}}^\infty$, 
    \begin{align*}
        \operatorname{dist}(z, Q) \leq \sum_{j=1}^\infty (\rho_j-1)j^{-s}.
    \end{align*}
    It is thus sufficient to impose the admissibility condition
    \begin{align*}
        \sum_{j=1}^\infty (\varrho_j-1)j^{-s}< r
    \end{align*}
    for some $r\in (0, r_*)$ to ensure that $\mathcal E_{\boldsymbol{\varrho}}^\infty\subset \sigma_s^{-1}(\mathcal D_{\C})$. This leads to the condition in Lemma \ref{lemma: (b,r)-holomorphy}, commonly known as the $(b,r)$-admissibility condition \parencite{cohen_analytic_2011} with $b_j=  j^{-s}$.
\end{remark}

The following lemma is a special case of \parencite[Lemma 2.2.7]{zech_sparse-grid_2018} with $b_j = j^{-s}$, adapted to the current problem setting. 

\begin{lemma}\label{lemma: (b,r)-holomorphy}
    Let $s>1$, $t>0$ and $L>0$. Let $\mathcal X_{\C}$, $\mathcal Y_{\C}$ be complex separable Hilbert spaces, let $\mathcal O_{\C}\subset \mathcal X_{\C}$ be open and let $G\in \mathcal H^{s,t}(L; \mathcal Y_{\C})$. Define $F = G \circ \sigma_s:Q\to \mathcal Y$. Then, there exists an $r>0$, depending only on $\mathcal D_{\C}$ and $\sigma_s$, independent of $G$, such that the following statements hold:
    
    \begin{enumerate}
        \item $F$ is continuous on $Q$. \label{item: continuous}
        \item For every \textit{admissible} sequence $\varrho= (\varrho_{j})_{j\in \N}\in (1, \infty)^{\N}$ satisfying
        \begin{align}\label{eq: admissible ellipse}
            \sum_{j=1}^{\infty} j^{-s} (\varrho_j-1)< r,
        \end{align}
        the map $F$ admits a separately analytic\footnote{i.e. analytic for any finite selection $(z_i)_{i\in I}$ of variables with $I \subset \N$ and $|I|<\infty$} extension to $\mathcal E^\infty_{\varrho} = \bigtimes_{j\in \N} \mathcal E_{\varrho_j}$.\label{item: (b,r)-admissibility}
        \item The extension $F: S_{s,r}\to \mathcal Y_{\C}$ with
        \begin{align}\label{eq: extension domain}
            S_{s,r} = \bigcup_{\{\varrho: \varrho \text{ is } \text{admissible}\}} \mathcal E^\infty_\varrho
        \end{align}
        is well-defined and satisfies $\sup_{z\in S_{s,r}}\|F(z)\|_{\mathcal Y^t_{\C}}\leq  L$. \label{item: uniform bound}
    \end{enumerate}
\end{lemma}

\begin{proposition}\label{proposition: extension to infdim ellipse}
    Let $s>1$ and let $S_{s,r}$ be defined as in \eqref{eq: extension domain}. For every $r>0$, the set $S_{s,r}$ is an open and bounded subset of $\ell^1_{-s}(\N)_{\C}$. Moreover, there exist $r', \widetilde r>0$ such that $r'<\widetilde r$ and
    \begin{align*}
        Q \subset S_{s,r'}\subset \sigma_s^{-1}(\mathcal D_{\C}) \subset S_{s,r"}
    \end{align*}
\end{proposition}
The proof of this proposition is given in Appendix \ref{appendix: proof of extension}.

This previous construction allows us to freely switch between two equivalent parametrizations of the input parameter, either in the \textit{coefficient representation} of $c \in Q$ or in the \textit{Hilbert-space representation} of $a= \sigma_s(c)\in \mathcal X$. For this reason, we impose the following assumption on the distribution of the design measure $\gamma$.
\begin{assumption}\label{assumption: measure}
    Let $s>1$ and let $\sigma_s$ be defined with respect to a fixed Riesz basis $\Psi\subset \mathcal X$. Let 
    \begin{align*}
        \pi = \bigotimes_{j=1}^\infty \operatorname{Unif}([-1,1])
    \end{align*}
    be the product uniform probability measure on $Q = [-1,1]^{\mathbb N}$ and assume that $\gamma$ is a probability measure on $\mathcal X$ defined by the push-forward $\gamma \coloneqq (\sigma_s)_{\#}\pi$.
\end{assumption}

Since $\sigma_s: Q \to \sigma_s(Q)$ is bijective, the composition with $\sigma_s$ defines an isometric isomorphism
\begin{align}\label{eq: isometric isomorphism}
    U: L^2(\sigma_s(Q), \mathcal Y, \gamma)\to L^2(Q, \mathcal Y, \pi)
\end{align}
with $U G = G\circ \sigma_s$ and inverse $U^{-1}F = F \circ \sigma_s^{-1}$. Indeed, by a change of measure, 
\begin{align*}
    \|UG\|^2_{L^2(Q, \mathcal Y, \pi)} = \int_{Q}\|G(\sigma_s(c))\|_{\mathcal Y}^2 ~d\pi(c) = \int_{\sigma_s(Q)}\|G(x)\|_{\mathcal Y}^2 ~d\gamma(x) =  \|G\|_{L^2(\gamma, \mathcal Y)} ^2.
\end{align*}
In particular, $\operatorname{supp}(\gamma) = \sigma_s(Q)$ and since $Q$ is compact in $\ell^1_{-s}(\N)$ and $\sigma_s$ is continuous, $\operatorname{supp}(\gamma)$ is compact in $\mathcal X$.

As we will see in Section \ref{subsection: approximation rate} and Section \ref{section: lower bounds}, $\mathcal H^{s,t}(L; \mathcal Y)$ is the appropriate class of holomorphic operators both from an approximation-theoretic and statistical viewpoint. The parameter $s>1$ encodes the anisotropic dependence on the input variables and the compactness of $\sigma_s(Q)$ in $\mathcal X$. The parameter $t>0$ imposes additional regularity in the output space of the holomorphic operator. Such regularity conditions are common in the context of parametric PDEs.

Since $\mathcal Y^t$ embeds compactly into $\mathcal Y$, the operator class is pointwise relatively compact. Together with the equicontinuity, which follows from the holomorphy assumption, this yields relative compactness in the relevant Bochner-space topology by the Arzelà-Ascoli theorem, see Appendix~\ref{appendix: compactness} for more details. 

This compactness is essential for information-theoretic arguments based on packing and covering numbers. Without such compactness in the input and output space, uniformly consistent statistical estimation becomes impossible. This follows for example from \parencite[Theorem 5.7]{johnstone2019gaussian}, see also \parencite{ibragimov1977estimation, ibragimov1997some} for the original statement. 

\subsection{Approximation Rates}\label{subsection: approximation rate}

We next recall quantitative polynomial approximation rates for the class $\mathcal H^{s,t}(L;\mathcal Y_{\C})$ from \cite{herrmann_neural_2024,reinhardt2026statistical}. The sparsity of
the Legendre polynomial chaos expansion is implied by the holomorpy
established in Lemma \ref{lemma: (b,r)-holomorphy}. Such results were
first proven by \textcite{cohen_analytic_2011}.

Denote by $(L_j)_{j\in \N}$ the $L^2([-1,1])$-normalized Legendre polynomials with $L_0 =1$, which are an orthonormal basis of $L^2([-1,1])$. It is well-known, that 
\begin{align}\label{eq: Legendre Upper Bound}
    \sup_{x\in [-1,1]} |L_j(x)| \leq \sqrt{1 + 2j}\quad \forall j \in \N_0.
\end{align}
Let $\mathcal F$ be the set of finitely supported multi-indices
\begin{align*}
    \mathcal F = \{\nu = (\nu_j)_{j\in \N_0}\in \N_0^{\N}: |\text{supp }\nu|<\infty\},
\end{align*}
with $\text{supp }\nu = \{j: \nu_j \neq 0\}$. For $c \in Q = [-1,1]^{\N}$ and $\nu \in \mathcal F$, define the tensorized Legendre polynomials by
\begin{align}\label{eq: tensorized Legendre}
    L_\nu(c) \coloneqq \prod_{j=1}^{\infty} L_{\nu_j}(c_j).
\end{align}
Then $(L_\nu)_{\nu \in \mathcal F}$ is an orthonormal basis of $L^2(Q, \pi)$. For $\nu \in \mathcal F$, define 
\begin{align}\label{eq: Legendre normalization}
    w_\nu = \prod_{j=1}^\infty (1 + 2\nu_j).
\end{align}

Define the $L^\infty([-1,1])$-normalized Legendre polynomials $(P_j)_{j\in \N}$ and normalize by the upper bound \eqref{eq: Legendre Upper Bound}, i.e  $P_j = L_j/(\sqrt{1 + 2j})$. The tensorized version $(P_\nu)_{\nu \in \mathcal F}$ are therefore defined by $P_\nu = w_\nu^{-1/2} L_\nu$.

By the isometric isomorphism $L^2(\mathcal X, \gamma) \cong L^2(Q, \pi)$, one may equivalently expand with respect to $(L_\nu)_{\nu \in \mathcal F}$ in $L^2(Q,  \pi)$ or with respect to $L_\nu \circ \sigma_s^{-1}$ in $L^2(\operatorname{supp}(\gamma),  \gamma)$. In particular, we will see that by the holomorphy of $G\in \mathcal H^{s,t}(L; \mathcal Y_{\C})\subset L^\infty(\operatorname{supp}(\gamma), \mathcal Y)\subset L^2(\gamma, \mathcal Y)$, the map $F = G\circ \sigma_s : Q \to \mathcal Y$ admits a unique expansion in $L^2(Q, \mathcal Y, \pi)$ by
\begin{align*}
    F(c) = \sum_{\nu \in \mathcal F} g_\nu L_\nu = \sum_{\nu \in \mathcal F} g_\nu^\infty P_\nu. 
\end{align*}
For each $\nu \in \mathcal F$, define the $\mathcal Y$-valued Legendre coefficient 
\begin{align}\label{eq: Y-valued Legendre coefficients}
  g_\nu = \int_Q L_\nu(c) F(c) ~d\pi(c) = \int_{\sigma_s(Q)} L_\nu(\sigma_s^{-1}(x)) G(x)~d\gamma(x), \quad
    g_\nu^\infty = w_\nu^{1/2}g_\nu.
\end{align}
and with respect to the Riesz basis $\Phi = (\phi_j)_{j\in \N}$ of $\mathcal Y$, 
\begin{align}\label{eq: scalar valued Legendre coefficients}
    g_{\nu,j} = \left \langle\int_Q L_\nu(c)  F(c)   ~ d\pi(c), \widetilde \phi_j\right\rangle_{\mathcal Y} = \int_Q L_\nu(c)  \left \langle F(c), \widetilde \phi_j\right\rangle_{\mathcal Y}   ~ d\pi(c)
\end{align}
Then 
\begin{align*}
    F(c) &= \sum_{\nu \in \mathcal F} g_\nu L_\nu(c)  = \sum_{j= 1}^\infty \sum_{\nu \in \mathcal F}  g_{\nu,j} L_\nu(c) \phi_j \\
    &= \sum_{\nu \in \mathcal F} g_\nu^\infty P_\nu(c) = \sum_{j= 1}^\infty \sum_{\nu \in \mathcal F}  g_{\nu,j}^\infty P_\nu(c) \phi_j 
\end{align*}
and 
\begin{align}\label{eq: polynomial chaos expansion}
    G(x) &= \sum_{\nu \in \mathcal F} g_\nu L_\nu(\sigma_s^{-1}(x)) =\sum_{j= 1}^\infty \sum_{\nu \in \mathcal F}  g_{\nu,j} L_\nu(\sigma_s^{-1}(x)) \phi_j\\
    &=\sum_{\nu \in \mathcal F} g_\nu^\infty P_\nu(\sigma_s^{-1}(x)) = \sum_{j= 1}^\infty \sum_{\nu \in \mathcal F}  g^\infty_{\nu,j} P_\nu(\sigma_s^{-1}(x)) \phi_j.
\end{align}

The following result is a version of \parencite[Theorem 5]{herrmann_neural_2024} or \parencite[Theorem 54]{reinhardt2026statistical} which themselves are reformulations of \parencite[Theorem 2.2.10]{zech_sparse-grid_2018}. 

\begin{theorem}\label{theorem: coefficient bounds}
  Let $s>1$, $t>0$, $L>0$, $t'\in [0,t)$, and fix 
  $\vartheta>0$ and 
  $p\in (1/s, 1)$.
  Let Assumption \ref{assumption: measure} hold. Let $\mathcal X_{\C}$, $\mathcal Y_{\C}$,  $\mathcal O_{\C}$,  $\mathcal D_{\C}$ and $\Psi = (\psi_j)_{j\in \N}\subset \mathcal X$ as in Definition \ref{def: holomorphic operator}. 
  Then there exist a constant $C>0$ and a non-negative sequence $(a_{\nu})_{\nu \in \mathcal F}\in \ell^p(\mathcal F)$ such that 
  for all $G\in \mathcal H^{s,t}(\mathcal O_{\C}, \mathcal D_{\C}, \mathcal Y_{\C}, L)$
    \begin{enumerate}
        \item For each $\nu \in \mathcal F$, the Legendre coefficients $g_\nu$ satisfy\label{item: 1}
        \begin{align}\label{eq: coefficient bound}
            w_\nu^\vartheta \left\Vert g_\nu\right\Vert_{\mathcal Y^{t'}_{\C}} \leq C a_\nu.
        \end{align} 
        \item There exists an enumeration $(\nu_i)_{i\in \N}$, such that $(a_{\nu_i})_{i\in \N}$ is monotonically decreasing and the set $\Lambda_N \coloneqq\{\nu_i: i \leq N\}\subset \mathcal F$ is  downward closed \footnote{A set $\Lambda \subset \mathcal F$ is called \textit{downward closed}, if $v\in \Lambda$ is implied by $\mu\in \Lambda$ and $\nu \leq \mu$ (element-wise)}. \label{item: 2}
        \item The expansion 
        \begin{align*}
             G(\sigma_s(c)) = \sum_{\nu \in \mathcal F} g_\nu L_\nu(c) 
        \end{align*}
        holds in $\mathcal Y^{t'}$ for all $c \in Q$ with absolute and uniform convergence.\label{item: 3}
    \end{enumerate}
\end{theorem}

From the $\ell^p(\mathcal F)$-summability of $(a_\nu)_{\nu \in \mathcal F}$, one can derive best $N$-term approximation rates and a priori downward closed index sets $\Lambda_N$ that can be chosen independently of a specific $G$. We will use these index sets below to construct approximation of the operator $G$, as well as projection estimators in Section \ref{section: upper bounds}. Indeed, for the enumeration $(\nu_i)_{i\in \N}$ in item~\ref{item: 2} of Theorem~\ref{theorem: coefficient bounds}, 
\begin{align*}
    a_{\nu_i} \lesssim i^{-1/p}.
\end{align*}
Applying Stechkin's Lemma \parencite{stechkin1951absolute} then yields the $L^2$-approximation rate $N^{-(1/p - 1/2)}$ and $L^\infty$-approximation rate $N^{-(1/p-1)}$. 

To be more precise, fix $N\in \N$ and let $\Lambda_N$ be an index set of cardinality $N$. Define $(g_\nu)_{\nu \in \mathcal F}$ as in \eqref{eq: Y-valued Legendre coefficients} and assume that $\mathcal Y = \R$. The truncation of the expansion \eqref{eq: polynomial chaos expansion}
corresponding to the index set $\Lambda_N$ is defined as
\begin{align}\label{eq: Legendre projection}
    g_N(x) \coloneqq \sum_{\nu \in \Lambda_N} g_\nu L_\nu(\sigma_s^{-1}(x))= \sum_{\nu \in \Lambda^\infty_N} g_\nu^\infty P_\nu(\sigma_s^{-1}(x)),
\end{align}
where $x\in \operatorname{supp}(\gamma)$. For holomorphic functionals, this yields the following approximation result \cite{cohen_analytic_2011,chkifa_breaking_2015,cohen2015approximation}.

\begin{theorem}[Rates of approximation for holomorphic functionals]\label{theorem: realvalued holomorphic approximation}
    Let $s>1$,  $L>0$, and let $\mathcal X_{\C}$,   $\mathcal O_{\C}$,  $\mathcal D_{\C}$ and $\Psi = (\psi_j)_{j\in \N}\subset \mathcal X$ as in Definition \ref{def: holomorphic operator}.
    Fix $\delta>0$ and let $\gamma$ be as in Assumption \ref{assumption: measure}.
    Then there exist constants $C_2$, $C_\infty>0$ such that, for every $N\in \N$, there exists a downward closed index set $\Lambda_N$ of cardinality $N$, and for every $g_0\in \mathcal H^{s}(\mathcal O_{\C}, \mathcal D_{\C}, \C, L)$ it holds
    \begin{align}\label{eq: L2 functional approximation}
        \lVert g_0 - g_N\rVert_{L^2(\gamma)}\leq C_{2} N^{-(s - 1/2 ) + \delta}
    \end{align}
    and
    \begin{align}\label{eq: Linfty functional approximation}
        \lVert g_0 - g_N\rVert_{\infty}\leq C_{\infty} N^{-(s - 1) + \delta}.
    \end{align}
\end{theorem}

For $\mathcal Y$-valued operators, we combine the polynomial approximation with a truncation in the output Riesz basis $\Phi$ of $\mathcal Y$. Recall that $(g_{\nu,j})_{\nu \in \mathcal F, j\in \N}$ denotes the corresponding Legendre coefficients in the generalized polynomial chaos expansion of any element $G\in L^2(\gamma, \mathcal Y)$ by 
\begin{align*}
    G(x) = \sum_{j=1}^\infty \sum_{\nu\in \mathcal F}    g_{\nu,j} L_\nu(\sigma_s^{-1}(x)) \phi_j, \qquad x\in \operatorname{supp}(\gamma).
\end{align*}
To truncate this expansion, we allow the truncation of the inner sum to also depend on $j$, i.e. to consider index sets of the form $(\Lambda_{N,j})_{j\leq N} \subset \mathcal F$ as in \parencite{herrmann_neural_2024}. For index sets $(\Lambda_{N,j})_{j\leq N},(\Lambda^\infty_{N,j})_{j\leq N} \subset \mathcal F$ and $x\in \operatorname{supp}(\gamma)$, we define
\begin{align}\label{eq: L2 gen Legendre projection}
    G_N (x) = \sum_{j=1}^N \sum_{\nu \in \Lambda_{N,j}} g_{\nu,j} L_\nu(\sigma_s^{-1}(x)) \phi_j
\end{align}
and 
\begin{align}\label{eq: Linfty gen Legendre projection}
    G_N^\infty (x) = \sum_{j=1}^N \sum_{\nu \in \Lambda^\infty_{N,j}} g_{\nu,j}^\infty P_\nu(\sigma_s^{-1}(x)) \phi_j.
\end{align}

Combined with the weighted summability results for the Legendre coefficients from Theorem~\ref{theorem: coefficient bounds}, bounds may be reformulated into a bound for the Legendre coefficients $g_{\nu,j}$. This yields the following approximation result
which can be found in \parencite[Theorem 6]{herrmann_neural_2024}: More precisely, \parencite[Theorem 6]{herrmann_neural_2024} gives a convergence rate in terms of a neural network approximation. Since the proof reapproximates a Legendre expansion with neural networks, the resulting error bound also holds for the Legendre expansion.

\begin{theorem}\label{theorem: operator approximation rate}
    Let $s>1$, $t>0$, $L>0$ and let $\mathcal X_{\C}$, $\mathcal Y_{\C}$,  $\mathcal O_{\C}$,  $\mathcal D_{\C}$ and $\Psi = (\psi_j)_{j\in \N}\subset \mathcal X$ as in Definition \ref{def: holomorphic operator} . Additionally, let $G_0\in \mathcal H^{s,t}(L; \mathcal Y_{\C})$ and suppose that Assumption \ref{assumption: measure} holds. For every $\delta>0$, there exist constants $C_2, C_\infty >0$ (depending on $s$, $t$, $L$, but independent of $G_0$) such that, for every $N\in \N$, there are downward closed sets $(\Lambda_{N,j})_{j\leq N}$ and $(\Lambda^\infty_{N,j})_{j\leq N}$ satisfying $\sum_{j=1}^\infty |\Lambda_{N,j}|\leq N$ and $\sum_{j=1}^\infty |\Lambda^\infty_{N,j}|\leq N$. For these index sets, it holds that
    \begin{align*}
        \lVert G_N - G_{0} \rVert _{L^{2}(\gamma, \mathcal Y)} \leq C_{2}N^{-\min \{s-1/2,t\} + \delta }
    \end{align*}
and 
    \begin{align*}
        \lVert  G_{N}^\infty - G_{0} \rVert _{\infty, \mathcal Y}\leq C_{\infty}N^{-\min \{s-1,t\} + \delta }.
    \end{align*}
\end{theorem}

\section{Metric Entropy Upper and Lower Bounds}\label{section: metric entropy}

We first establish quantitative metric entropy estimates for the class of holomorphic functionals $\mathcal H^{s}(L; \C)$ and the class of holomorphic operators $\mathcal H^{s,t}(L; \mathcal Y_{\C})$, in both the corresponding $L^2$- and $L^\infty$-losses. The upper bounds will follow from the polynomial approximation results and the Legendre coefficient estimates of Section \ref{subsection: approximation rate}, whereas the lower bounds are obtained from the construction of explicit packings of the corresponding classes.

The relative compactness of $\mathcal H^{s,t}(L; \mathcal Y_{\C})$ with respect to $L^\infty (\operatorname{supp}(\gamma), \mathcal Y)$ and $L^2(\gamma, \mathcal Y)$ is proved in Theorem \ref{appendix: compactness} and ensures that the metric entropy is finite. Although the metric entropy estimates in Proposition \ref{proposition: L2 metric entropy UB}, Proposition \ref{proposition: Linfty metric entropy UB}, Proposition \ref{proposition: L2 metric entropy LB} and Proposition \ref{proposition: Linfty metric entropy LB} are of independent interest, they will be used to derive the minimax lower bounds in Section \ref{subsection: LB holomorphic functionals}, \ref{subsection: GWN lower bounds} and \ref{subsection: RWN lower bounds}.

\subsection{Metric Entropy Upper Bounds}

We begin with the $L^2$ metric entropy upper bounds. The construction of the covering is performed in two steps. First, we truncate the $\mathcal Y$-valued Legendre polynomial chaos expansion using an a priori chosen index set $\Lambda_N$, such that the squared $L^2$-error is at most $\varepsilon^2/2$. We then cover each of the $N$ coefficients with balls of radii $\varepsilon_1, \dots, \varepsilon_N$. By Theorem \ref{theorem: coefficient bounds}, we have explicit quantitative control of the size of each of the coefficients. Since the tensorized Legendre polynomials form an orthonormal basis, the resulting $L^2$-error from discretizing the coefficients is bounded by $\sum_{i=1}^N \varepsilon_i^2$. Any choice of $\varepsilon_1, \dots, \varepsilon_N$ such that this sum is upper bounded by $\varepsilon^2/2$ then produces an $\varepsilon$-covering. Optimizing these $\varepsilon_i$ yields the stated $L^2$ entropy bounds.

\begin{proposition}[$L^2$-metric entropy upper bounds]\label{proposition: L2 metric entropy UB}
    Let $s>1$, $t>0$ and $L>0$. Suppose that Assumption \ref{assumption: measure} holds. Then, for every $\tau>0$ and all sufficiently small $\varepsilon>0$ (cf.~\eqref{eq:covering})
    \begin{equation*}
        \log N(\varepsilon, \mathcal H^s(L; \C), \|\cdot\|_{L^2(\gamma)}) \lesssim \varepsilon^{-\frac{1}{s-1/2} - \tau}.
    \end{equation*}
    Moreover, setting $\kappa_2 = \min \{s-1/2, t\}$, we have
    \begin{equation*}
         \log N(\varepsilon, \mathcal H^{s,t}(L; \mathcal Y_{\C}), \|\cdot\|_{L^2(\gamma, \mathcal Y)}) \lesssim \varepsilon^{-\frac{1}{\kappa_2} - \tau}.
    \end{equation*}
    The implicit constants are independent of $\varepsilon$, but may depend on parameters $s, t, L, \Psi, \mathcal O_{\C}$ and $\mathcal D_{\C}$ which determine $\mathcal H^{s}(L; \C)$ and $\mathcal H^{s,t}(L; \mathcal Y_{\C})$.
\end{proposition}
\begin{proof}
We split the proof of the proposition into two parts, first proving the upper bounds for the class of holomorphic functionals $\mathcal H^s(L, \C)$ and then for the class of holomorphic operators $\mathcal H^{s,t}(L, \mathcal Y_{\C})$.
    \paragraph{Metric entropy upper bounds for $\mathcal H^s(L;\C)$.}
    Let $\delta>0$ be sufficiently small and set $s' = s- \delta < s$, such that $\frac{1}{s'-1/2}\leq \frac{1}{s-1/2}+ \tau$. Let $g\in \mathcal H^s(L;, \C)$ and let $g_\nu$ denote its scalar-valued Legendre coefficients which are defined as in \eqref{eq: Y-valued Legendre coefficients}. 
    By Theorem \ref{theorem: coefficient bounds}, there exists a constant $C>0$ and an enumeration $(\nu_i)_{i\in \N}$, both independent of $g\in \mathcal H^s(L;, \C)$, such that $|g_{\nu_i}|\leq Ci^{-s'}$ for all $i\in \N$. For the index set $\Lambda_N = \{\nu_1, \dots, \nu_N\}\subset \mathcal F$, Theorem \ref{theorem: realvalued holomorphic approximation} and Parseval's identity imply
    \begin{equation}\label{eq:sumgnubound}
      \left\|g - \sum_{\nu \in \Lambda_N} g_\nu L_\nu \circ(\sigma_s)^{-1}\right\|_{L^2(\mathcal X, \gamma)}^2  =\left\|g \circ \sigma_s - \sum_{\nu \in \Lambda_N} g_\nu L_\nu \right\|_{L^2(Q, \pi)}^2
      =\sum_{\nu\notin\Lambda_N} g_\nu^2
        \leq C_{2} N^{-(2s'-1)},
    \end{equation}
    where we used \eqref{eq: isometric isomorphism} in the first equality.
    Choose $N \asymp \varepsilon^{-\frac{2}{2s'-1}}$ such that the last expression is smaller than $\varepsilon^2/2$.
    It remains to discretize the coefficients $\{g_{\nu_i}: i= 1, \dots, N\}$. 
    For $i= 1, \dots, N$, cover the interval $[-Ci^{-s'}, Ci^{-s'}]$ by intervals of length $\varepsilon_i>0$, with the corresponding covering number satisfying
    \begin{equation}\label{eq:upboundNi}
        N_i \coloneqq N(\varepsilon_i, [-Ci^{-s'}, Ci^{-s'}], |\cdot|)  \lesssim \left(\frac{i^{-s'}}{\varepsilon_i}\right).
    \end{equation}
    Since the Legendre polynomials $L_\nu \circ \sigma_s^{-1}$ are orthonormal in $L^2(\mathcal X, \gamma)$, it suffices to require that 
    \begin{equation*}
        \sum_{i=1}^N \varepsilon_i^2 \leq \frac{\varepsilon^2}{2}.
    \end{equation*}
    Up to constants, minimizing the upper bound of the cardinality $\prod_{i=1}^N N_i$ of the covering is equivalent to minimizing the sum of the logarithms, subject to the given constraint on the $\varepsilon_i$, i.e. 
    \begin{equation*}
        \min_{\varepsilon_1, \dots, \varepsilon_N} \sum_{i=1}^N \log\left(\left(\frac{i^{-s'}}{\varepsilon_i}\right)\right), \quad \text{s.t.} \sum_{i=1}^N \varepsilon_i^2 \leq \frac{\varepsilon^2}{2}.
    \end{equation*}
    The resulting minimization problem is solved for $\varepsilon_i = \varepsilon/\sqrt{2N}$. For each $i=1, \dots, N$, let $\mathcal C_i \coloneqq \{g_{\nu_i, j_i}: j_i =1, \dots, N_i\}$ be an $\varepsilon_i$-net for the $i$-th coefficient. Hence, for each $g\in \mathcal H^{s}(L; \C)$ and for each $i=1, \dots, N$, there exist a $g_{\nu_i, j_i} \in \mathcal C_i$, such that with \eqref{eq:upboundNi}
    \begin{equation*}
      \left\|g - \sum_{i=1}^N g_{\nu_i, j_i} L_\nu \circ(\sigma_s)^{-1}\right\|_{L^2(\gamma)}^2
      \le \sum_{i=1}^N \varepsilon_i^2+\sum_{\nu\notin\Lambda_N}g_{\nu}^2\le
      \varepsilon^2.
    \end{equation*}
    By the choice of $N$, it holds that $\varepsilon/\sqrt N \asymp N^{-s'}$ and $i^{-s'}/\varepsilon_i \asymp (N/i)^{s'}$.
    Applying the elementary inequality $\sum_{j=1}^N \log(j) \geq N \log(N) - N$, (which follows from Stirling's approximation),
    \begin{align}\label{eq: sterling}
        \log N(\varepsilon, \mathcal H^s(L; \C), \|\cdot\|_{L^2(\gamma)}) &\leq \sum_{i=1}^N\log(N_i) \\
        &\lesssim \sum_{i=1}^N \Big(\log\left(\left(\frac{N}{i}\right)^{s'}\right)\Big)\nonumber\\
        &\leq s'\Big(N\log(N)- \sum_{i=1}^N\log(i)\Big)\nonumber\\
        &\lesssim N\\
        &\lesssim \varepsilon^{-\frac{1}{s-1/2} - \tau}\nonumber,
    \end{align}
    which proves the metric entropy upper bound for $\mathcal H^{s}(L; \C)$.
    \paragraph{Metric entropy upper bounds for $\mathcal H^{s,t}(L;\mathcal Y_\C)$.} The idea for the construction of the covering is the same as in the scalar-valued case, so we only describe the differences in the proof to the first case here. Fix $\tau>0$ and choose $\delta>0$ sufficiently small and set $s' = s- \delta$ and $t' = t- \delta$ so that $s'>1, t'>0$ and 
    \begin{equation*}
        \frac{1}{\min\{s'-1/2, t'\}} \leq \frac{1}{\min\{s-1/2,t\}} + \tau.
    \end{equation*}
    By Theorem \ref{theorem: coefficient bounds}, after enumerating $(\nu_i)_{i\in \N}$, there exists a constant $C>0$ and $\mathcal Y$-valued Legendre coefficients $g_\nu$ satisfying
    \begin{equation*}
        \|g_{\nu_i}\|_{\mathcal Y^{t'}} \leq C i^{-s'}.
    \end{equation*}
    As for the holomorphic functionals, we truncate the Legendre polynomial chaos expansion after $N$ terms and choose $N\asymp \varepsilon^{-\frac{2}{2s'-1}}$ such that the squared $L^2$ truncation error is smaller than $\varepsilon^2/2$. For each $i = 1, \dots, N$, define $B_{s',t', i} = \{y\in \mathcal Y, \|y\|_{\mathcal Y^{t'}}\leq Ci^{-s'}\}$, which are balls in $\mathcal Y^{t'}$ of radius $Ci^{-s'}$. Since balls in $\mathcal Y^{t'}$ are ellipsoids with respect to the $\|\cdot\|_{\mathcal Y}$-norm, for which the metric entropy is known \parencite{kolmogorov1959entropy}, 
    \begin{equation}\label{eq: metric entropy coefficient operator}
        \log(N_i) \coloneqq  \log N(\varepsilon_i, B_{s',t', i}, \|\cdot\|_{\mathcal Y} ) \asymp \left(\frac{i^{-s'}}{\varepsilon_i}\right)^{1/t'}.
    \end{equation}
    As in the scalar-valued case, it suffices that we require $\sum_{i=1}^N \varepsilon_i^2 \leq \varepsilon^2/2$. Let $\mathcal C_i = \{g_{\nu_i, j_i}: j_i = 1, \dots N_i\}$ denote a minimal $\varepsilon$-covering of $B_{s',t',i}$ and set 
    \begin{equation*}
        \mathcal C = \left\{ \sum_{i=1}^N g_{\nu_i, j_i}L_{\nu_i}\circ \sigma_s^{-1}: j_i =1, \dots, N_i\right\}.
    \end{equation*}
    Then \begin{equation}\label{eq: coveringcardinality}
        |\mathcal C|\leq \prod_{i=1}^N N_i.
    \end{equation}
    
    For each $i = 1, \dots, N$, there exists a specific $k_i \in \{1, \dots, N_i\}$, such that $\|g_{\nu_i} - g_{\nu_i, k_i}\|_{\mathcal Y} \leq \varepsilon_i$. By the orthonormality of $(L_\nu \circ \sigma_s^{-1})_{\nu \in \mathcal F}$ in $L^2(\gamma)$ and since $L^2(\gamma, \mathcal Y) \simeq L^2(\gamma)\otimes \mathcal Y$, it holds that 
    \begin{equation*}
        \left\|G - \sum_{i=1}^N g_{\nu_i, k_i}L_{\nu_{i}}\circ \sigma_s^{-1}\right\|_{L^2(\gamma, \mathcal Y)}^2 \leq \sum_{\nu \not \in \Lambda_N} \|g_\nu\|_{\mathcal Y}^2 + \sum_{i=1}^N \varepsilon_i^2 \leq \varepsilon^2.
    \end{equation*}
    
    Minimizing the right-hand side of \eqref{eq: coveringcardinality} is equivalent to the minimization problem \begin{equation}\label{eq: minimisation L2}
        \min_{\varepsilon_1, \dots, \varepsilon_N} \sum_{i=1}^N \left(\frac{i^{-s'}}{\varepsilon_i}\right)^{1/t'}, \quad \text{s.t.} \sum_{i=1}^N \varepsilon_i^2 \leq \frac{\varepsilon^2}{2}.
    \end{equation}
    Set 
    \begin{equation*}
        S_N \coloneqq \sum_{i=1}^N i^{-\frac{2s'}{1+ 2t'}}.
    \end{equation*}
    Using Lagrange-multipliers and applying Lemma \ref{lemma: constr opt prob L2}, with $\alpha = s'$ and $\beta = t'$, the minimization problem is solved by
    \begin{equation*}
        \varepsilon_i = \frac{\varepsilon i^{-\frac{s'}{1+2t'}}}{\sqrt{2S_N}}.
    \end{equation*}
    
    Plugging the optimal choice of $\varepsilon_i$ into \eqref{eq: minimisation L2} (see also Lemma \ref{lemma: constr opt prob L2}) and using $N\asymp \varepsilon^{-\frac{2}{2s'-1}}$ yields 
    \begin{equation*}
        \sum_{i=1}^N \left(\frac{i^{-s'}}{\varepsilon_i}\right)^{1/t'} = \left(\frac{\sqrt{2}}{\varepsilon}\right)^{1/t'} S_N^{\frac{1+ 2t'}{2t'}}.
    \end{equation*}
    By integral comparison, 
    \begin{align*}
        S_N \lesssim \begin{cases}
            N^{1 - \frac{2\alpha}{1+2\beta}}     &2\alpha <1+2\beta\\
            \log(N)                     &2\alpha = 1+2\beta\\
            1                           &2\alpha > 1+ 2\beta
        \end{cases},
    \end{align*}
    and it follows that the metric entropy upper bound satisfies
    \begin{equation*}
        \log N(\varepsilon, \mathcal H^{s,t}(L; \mathcal Y_{\C}), \|\cdot\|_{L^2(\gamma, \mathcal Y)}) \leq \log(|\mathcal C|) \leq \sum_{i=1}^N \log(N_i) \lesssim \varepsilon^{-\frac{1}{\kappa_2} -\tau}.
    \end{equation*}
\end{proof}
The proof of the $L^\infty$-metric entropy bounds follows the same truncate-then-discretize strategy as in the $L^2$-case, with two main differences. Firstly, we use an expansion w.r.t. the $L^\infty$-normalized Legendre polynomials $P_\nu$ with Legendre coefficients $g_\nu^\infty \in \mathcal Y$, such that the corresponding $N$-term truncation converges with rate by $s-1$. Secondly, we use the fact that $\|P_\nu\|_\infty \leq 1$ and the triangle inequality for the discretization error for the Legendre coefficients instead of the orthonormality of $(L_\nu \circ \sigma_s^{-1})_{\nu \in \mathcal F}$ in $L^2(\gamma)$ and Parseval's identity. Apart from these differences, the construction of the covering and the optimization of the upper-bounds follows essentially the same steps as the $L^2$-case. The details of the proof can be found in Appendix~\ref{appendix: proof of Linfty metric entropy UB}.

\begin{proposition}[$L^\infty$-metric entropy upper bounds]\label{proposition: Linfty metric entropy UB}
    Let $s>1$, $t>0$ and $L>0$ and let $\tau>0$ be arbitrarily small. Then 
    \begin{equation*} 
        \log N\left(\varepsilon, \mathcal H^s(L; \C), \|\cdot\|_{\infty}\right) \lesssim \varepsilon^{-\frac{1}{s-1} - \tau}. 
    \end{equation*} 
    Moreover, setting $\kappa_{\infty} = \min\{s-1,t\}$, we have 
    \begin{equation*}
        \log N\left(\varepsilon, \mathcal H^{s,t}(L; \mathcal Y_{\C}), \|\cdot\|_{\infty, \mathcal Y}\right) \lesssim 
        \varepsilon^{-\frac{1}{\kappa_{\infty}} - \tau}.
    \end{equation*} 
\end{proposition}

\subsection{Metric Entropy Lower Bounds}

We next derive matching lower bounds by constructing explicit packings of the classes of holomorphic functionals and holomorphic operators. 
Importantly, we will work with a class of maps restricted to (affine) linear operators. It turns out that this class already suffices to match corresponding upper bounds. To formalize this, set
\begin{align*}
\mathcal H_{\rm lin}^{s,t}(\mathcal O_{\C}, \mathcal D_{\C}, \mathcal Y_{\C}, L):=
\{G\in \mathcal H^{s,t}(\mathcal O_{\C}, \mathcal D_{\C}, \mathcal Y_{\C}, L)\,:\,G\text{ is affine linear}\},
\end{align*}
and analogous to Definitions \ref{def: holomorphic operator} and \ref{definition: holomorphic functional}, we also use the shorter notations
\begin{equation*}
\mathcal H_{\rm lin}^{s,t}(L;\mathcal{Y}_{\C}),\qquad
\mathcal H_{\rm lin}^{s}(L;\C),
\end{equation*}
where the latter space corresponds to the linear operators with $\mathcal{Y}_{\C}=\C$ (analogous to Definition \ref{definition: holomorphic functional}).

\begin{proposition}[$L^2$-metric entropy lower bounds]\label{proposition: L2 metric entropy LB}
    Let $s>1, t>0$ and $L>0$ and fix $N\in \N$. Let $\gamma$ be as in Assumption \ref{assumption: measure}. 
    Then, for all sufficiently small $\varepsilon>0$, 
    \begin{align}\label{eq: lower bound metric entropy functional L2}
         \log N(\varepsilon, \mathcal H^s(L; \C), \|\cdot\|_{L^2(\gamma)})   \ge \log N(\varepsilon, \mathcal H_{\rm lin}^s(L; \C), \|\cdot\|_{L^2(\gamma)}) \gtrsim \varepsilon^{-\frac{1}{s-1/2}}.
    \end{align}
    Moreover, setting $\kappa_2 = \min\{s-1/2, t\}$, we have
    \begin{align}\label{eq: lower bound metric entropy operator L2}
        \log N(\varepsilon, \mathcal H^{s,t}(L; \mathcal Y_{\C}), \|\cdot\|_{L^2(\gamma)})  \ge \log N(\varepsilon, \mathcal H^{s,t}_{\rm lin}(L; \mathcal Y_{\C}), \|\cdot\|_{L^2(\gamma, \mathcal Y)}) \gtrsim  \varepsilon^{-\frac{1}{\kappa_2}}.
    \end{align}
    The implicit constants are independent of $\varepsilon$.
\end{proposition}
\begin{proof}
    We use a similar combinatorial argument for both the class of holomorphic functionals and the class of holomorphic operators. Based on the Varshamov--Gilbert lemma, see \parencite[Lemma 2.9]{tsybakov_introduction_2008} or \parencite[Example 3.1.4]{gine_mathematical_2021}, for every $N\in \N$, there exists a set $\Omega \subset \{-1,1\}^N$ with $|\Omega|\geq 3^{N/4}$, such that for any $\omega, \omega' \in \Omega$, the Hamming distance satisfies
    \begin{equation*}
        \rho(\omega, \omega') \coloneqq \sum_{i=1}^N \mathbbm 1_{\{\omega_i \neq \omega_i'\}} \geq \frac{N}{8}.
    \end{equation*}

    \paragraph{Metric entropy lower bounds for $\mathcal{H}_{\rm lin}^{s}(L;\C)$}
    For $\omega \in \Omega$, define 
    \begin{equation*}
        g_\omega(z) \coloneqq \delta \sum_{i=1}^N \omega_i ~\langle z, \widetilde\psi_i\rangle_{\mathcal X}, \qquad z \in \mathcal O_{\C},
    \end{equation*}
    where $\delta>0$ is some fixed constant which is chosen later. Clearly, $g_\omega$ is a continuous linear functional and hence holomorphic on $\mathcal O_{\C}$. To show that $g_\omega \in \mathcal H_{\rm lin}^{s}(L;\C)$, we need to show that $g_\omega$ is uniformly bounded on $\mathcal D_{\C}$. Since $\mathcal D_{\C}$ is bounded in $(\mathcal X_{\ell^1})_{\C}$, 
    \begin{equation*}
        R \coloneqq \sup_{z\in \mathcal D_{\C}} \|z\|_{\mathcal X_{\ell^1}}<\infty,
    \end{equation*}
    and therefore
    \begin{equation*}
        \sup_{z\in \mathcal D_{\C}}|g_\omega(z)| \leq \delta \sup_{z\in \mathcal D_{\C}}\sum_{i=1}^N\left|\langle z, \widetilde \psi_i\rangle_{\mathcal X}\right| \leq \delta R.
    \end{equation*}
    Choosing $\delta< L/R$ ensures that $g_\omega \in \mathcal H_{\rm lin}^{s}(L;\C)$ for any $\omega \in \Omega$ and $N\in \N$. 
    By \eqref{eq: input expansion}, 
    \begin{equation*}
        g_\omega (\sigma_s(c)) = \delta \sum_{i=1}^N \omega_i i^{-s} c_i, \qquad c\in Q.
    \end{equation*}
    For any $\omega, \omega' \in \Omega$ with $\omega \neq \omega'$, and since by Assumption \ref{assumption: measure}, it holds that $\gamma = (\sigma_s)_\#\pi$,
    \begin{align}\label{eq: L2 sep functional}
        \|g_\omega - g_{\omega'}\|_{L^2(\gamma)}^2 
        &= \delta^2 \|g_\omega\circ \sigma_s - g_{\omega'} \circ \sigma_s\|_{L^2(\pi)}\nonumber\\
        &= \frac{\delta^2}{3}\sum_{i=1}^N i^{-2s}(\omega_i - \omega_i')^2\nonumber \\
        &\geq \frac{\delta^2}{3} N^{-2s}\sum_{i=1}^N (\omega_i - \omega_i')^2\\
        & =  \frac{4\delta^2}{3}N^{-2s} \rho(\omega,\omega')\nonumber\\
        &\geq \frac{\delta^2}{6} N^{1-2s}.\nonumber
    \end{align}
    Here, we used that
    \begin{equation*}
        \int_{Q} c_i^2 ~d\pi(c) = \frac{1}{2}\int_{-1}^1 c_i^2 ~dc_i = \frac{1}{3}
    \end{equation*}
    in the second equality.
    Consequently, choosing $N \asymp \varepsilon^{-1/(s-1/2)}$ such that the last display equals $\varepsilon^2$ yields a packing in $\mathcal H^{s}(L; \C)$ and proves the lower bound \eqref{eq: lower bound metric entropy functional L2}. 

    \paragraph{Metric entropy lower bounds for $\mathcal H_{\rm lin}^{s,t}(L; \C)$.} Recall that $\kappa_2 = \min \{s-1/2, t\}$ and that $\Phi = (\phi_j)_{j\in \N}$ is a Riesz basis of $\mathcal Y$ with Riesz constants $D_\Phi> d_\Phi>0$. 
    
    Suppose first that $s-1/2 \leq t$, in which case $\kappa_2 = s-1/2$. Note that $d_\Phi \leq \|\phi_1\|_{\mathcal Y} = \|\phi_1\|_{\mathcal Y^{t}} \leq D_{\Phi}$. For $\omega \in \Omega$, define $G_\omega: \mathcal X_{\C} \to \mathcal Y_{\C}$ by
    \begin{equation*}
        G_\omega(z) \coloneqq \delta \sum_{i=1}^N \omega_i \langle z, \widetilde \psi_i\rangle_{\mathcal X} \phi_1.
    \end{equation*}
    As in the scalar-valued case, $G_\omega$ is holomorphic and we can bound
    \begin{equation*}
        \sup_{z\in \mathcal D_{\C}}\|G_\omega(z)\|_{\mathcal Y^{t}} \leq \delta R D_{\Phi}.
    \end{equation*}
    Choosing $\delta< L/RD_\Phi$ ensures that $G_\omega \in \mathcal H_{\rm lin}^{s,t}(L;\mathcal{Y}_{\C})$.
    Furthermore,  for $\omega$, $\omega'\in \Omega$, 
    \begin{equation*}
        G_\omega(z) - G_{\omega'}(z) = (g_\omega(z) - g_{\omega'}(z))\phi_1,
    \end{equation*}
    and thus, using $L^2(\gamma, \mathcal Y) \simeq L^2(\gamma)\otimes \mathcal Y$, the lower Riesz bound and \eqref{eq: L2 sep functional}
    \begin{equation*}
        \|G_\omega - G_{\omega'}\|^2_{L^2(\gamma, \mathcal Y)} = \|g_\omega - g_{\omega'}\|^2_{L^2(\gamma)} \|\phi_1\|_{\mathcal Y}^2 \geq \frac{d_\Phi^2 \delta^2}{6} N^{1-2s}.
    \end{equation*}
    Setting $N \asymp \varepsilon^{-1/(s-1/2)}$ such that last display is equal to $\varepsilon^2$, yields the desired metric entropy lower bound for $s-1/2 \leq t$.
    
    If $s-1/2>t$, it holds that $\kappa_2 = t$. Given $\omega \in \Omega$, define the constant operator $G_\omega: \mathcal X_{\C}\to \mathcal Y_{\C}$ by
    \begin{equation*}
        G_\omega(z) = \delta N^{-t-1/2}\sum_{j=1}^N \omega_j \phi_j,
    \end{equation*}
    which is holomorphic, and by the definition of the $\mathcal Y^t$-norm, 
    \begin{align*}
        \sup_{z\in \mathcal D_{\C}} \|G_\omega(z)\|_{\mathcal Y^t}^2
        & = \delta^2  N^{-2t-1}\left\|\sum_{j=1}^N \omega_j \phi_j\right\|_{\mathcal Y^t}^2\\
        &\leq \delta^2D_\Phi^2 N^{-2t-1} \sum_{j=1}^N j^{2t}\\
        &\leq \delta^2D_\Phi^2 N^{-2t-1}\left(1+ \int_1^N x^{2t}~dx\right)\\
        &\leq \delta^2D_\Phi^2,
    \end{align*}
    where we used an integral comparison bound in the second inequality. Choosing $\delta< L/( D_\Phi)$ ensures that $G_\omega\in\mathcal H_{\rm lin}^{s,t}(L;\mathcal Y_\C)$. Let $\omega, \omega'\in \Omega$ with $\omega \neq \omega'$. Since $\gamma$ is a probability measure and $G_\omega $ is constant in its input arguments, 
    \begin{align*}
            \|G_\omega-G_{\omega'}\|_{L^2(\gamma,\mathcal Y)}^2 
            &= \delta^2N^{-2t-1} \left\| \sum_{j=1}^N (\omega_j-\omega_j')\phi_j \right\|_{\mathcal Y}^2 \\ 
            &\geq d_\Phi^2\delta^2N^{-2t-1} \sum_{j=1}^N (\omega_j-\omega_j')^2 \\ 
            & = 4 d_\Phi^2\delta^2 \rho(\omega, \omega')\\
            &\geq \frac{d_\Phi^2\delta^2}{2} N^{-2t}.
    \end{align*}
    Choosing $N \asymp \varepsilon^{-1/t}$ such that last display is equal to $\varepsilon^2$, \eqref{eq: lower bound metric entropy operator L2} follows.
\end{proof}
The $L^\infty$-metric entropy lower bounds are obtained from the same packing arguments as for the $L^2$-metric entropy lower bound, with the difference that the separation is measured with respect to the supremum norm. We again distinguish the two cases in which either the input regularity or the output regularity determine the exponent $\kappa_\infty = \min \{s-1,t\}$. The exponent coincides (up to an arbitrarily small difference) with the exponent which appears in both the $L^\infty$-metric entropy upper bounds of Proposition \ref{proposition: Linfty metric entropy UB} and the $L^\infty$ approximation rate of Theorem \ref{theorem: operator approximation rate}.

The proof of the following proposition can be found in Appendix~\ref{appendix: proof of Linfty metric entropy LB}.
\begin{proposition}[$L^\infty$-metric entropy lower bounds]\label{proposition: Linfty metric entropy LB}
    Let $s>1$, $t>0$ and $L>0$, and suppose that Assumption \ref{assumption: measure} holds. Then, for all sufficiently small $\varepsilon>0$,
    \begin{align}\label{eq: Linfty metric entropy LB functional}
        \log N\left(\varepsilon, \mathcal H^s(L; \C), \|\cdot\|_{\infty}\right) \geq \log N\left(\varepsilon, \mathcal H^s_{\rm lin}(L; \C), \|\cdot\|_{\infty}\right)\gtrsim \varepsilon^{-\frac{1}{s-1}}.
    \end{align}
    Furthermore, setting $\kappa_{\infty} = \min \{s-1,t\}$, it holds that
    \begin{align}\label{eq: Linfty metric entropy lower bound operator}
        \log N\left(\varepsilon, \mathcal H^{s,t}(L; \mathcal Y_{\C}), \|\cdot\|_{\infty, \mathcal Y}\right) \geq \log N\left(\varepsilon, \mathcal H^{s,t}_{\rm lin}(L; \mathcal Y_{\C}), \|\cdot\|_{\infty, \mathcal Y}\right)
        \gtrsim \varepsilon^{-\frac{1}{\kappa_{\infty}}}.
    \end{align}
    The implicit constants are independent of $\varepsilon$.
\end{proposition}

\section{Information-Theoretic Lower Bounds}\label{section: lower bounds}

In the previous section, we derived upper and lower bounds on the metric entropy of the scalar-valued class $\mathcal H^s(L; \C)$ and the Hilbert space-valued class $\mathcal H^{s,t}(L; \mathcal Y_{\C})$. We now combine these metric entropy estimates with information-theoretic arguments to derive minimax lower bounds in both the operator white noise model, formally given by 
\begin{equation}\label{eq: white noise}
    \mathbb Y_{G_0}^{(n)} = G_0  + \frac{\sigma}{\sqrt{n}}\mathbb  W_{H},
\end{equation}
and the random design white noise model 
\begin{equation*}
    Y_i = G_0(X_i) + \sigma\varepsilon_i, \quad \sigma>0,~ i =1, \dots, n.
\end{equation*}

To lower bound the minimax risk, one first reduces the estimation problem to a multiple hypotheses testing problem, whose error will be bounded from below via information-theoretic arguments. Let $\mathcal M$ be a parameter class, possibly infinite-dimensional, equipped with the semi-metric $d$, let $f\in \mathcal M$ and $\P_f$ the corresponding observation law. The general proof uses the following two main steps:

\begin{enumerate}
    \item Choose a $2\delta$-packing $\{f_1, \dots, f_M\}\subseteq \mathcal M$ with respect to $d$ and write $\P_j = \P_{f_j}$. Let $J$ be uniformly distributed on $\{1, \dots, M\}$ and conditionally on $J=j$, let $Z\sim P_j$. Denote the mixture distribution by $\overline P\coloneqq \frac{1}{M}\sum_{j=1}^M P_j$. Then the mutual information between $Z$ and $J$ is given by \parencite[Ch.~15.3.1]{wainwright2019high}
    \begin{equation*}
        I(Z,J) = \frac{1}{M}\sum_{j=1}^M \operatorname{KL}(P_j, \overline P).
    \end{equation*}
    The standard reduction from estimation to testing, followed by Fano's inequality \parencite[Proposition 15.12]{wainwright2019high}, yields 
    \begin{align}\label{eq: Fano's lemma}
        \inf_{\widehat f}\sup_{f\in \mathcal M}\mathbb E_{f}[d(f,\widehat f)] \geq \delta\left(1 - \frac{I(Z, J) + \log 2}{\log M}\right).
    \end{align}
    The construction of such packings will be obtained directly from the metric entropy lower bounds established in Section \ref{section: metric entropy}.
    \item We upper-bound the mutual information using the Yang--Barron method \parencite{yang1999information} (used here in the form of \parencite[Lemma 15.21]{wainwright2019high}). Let $N_{KL}(\varepsilon, \mathcal M)$ denote the $\varepsilon$-covering number of $\{\P_f: f\in \mathcal M\}$ with respect to the square root of the Kullback--Leibler divergence. Then
    \begin{equation}\label{eq: yang-barron}
        I(Z,J) \leq \inf_{\varepsilon>0} (\varepsilon^2 + \log N_{KL}(\varepsilon, \mathcal M)).
    \end{equation}
\end{enumerate}
{For the generalized Gaussian white noise model, the relationship between the Kullback-Leibler divergence and the $H$-norm is recalled in Section 4.1. This allows us to control the KL-covering numbers directly by the metric entropy estimates derived in Section \ref{section: metric entropy}.}

For $\mathcal Y$-valued operators, we set $\mathcal M = \mathcal H^{s,t}(L; \mathcal Y_{\C})$, $H = L^2(\gamma, \mathcal Y)$ and $d$ is the metric corresponding to the norm $\|\cdot\|_{L^2(\gamma, \mathcal Y)}$ or $\|\cdot\|_{\infty, \mathcal Y}$. For scalar-valued functionals, we set $\mathcal M = \mathcal H^{s}(L; \C)$, $H = L^2(\gamma)$ and $d$ is the metric corresponding to the norm $\|\cdot\|_{L^2(\gamma)}$ or $\|\cdot\|_{\infty}$.

\subsection{Generalized Gaussian White Noise Models}\label{subsection: white noise}

We consider the following two observation models. First, following \textcite{reinhardt2026statistical}, we study the random design white noise model
\begin{align}\label{eq: random design white noise}
    Y_i = G_0(X_i) + \sigma\varepsilon_i,
\end{align}
with i.i.d. $X_i \sim \gamma$ and i.i.d. $\mathcal Y$-white noise processes $\varepsilon_i$. Motivated by the classical Gaussian white noise model \parencite{gine_mathematical_2021} on $L^2([-1,1])$, we consider an infinite-dimensional analogue on a general Hilbert space $H$.

It is well known that on any separable Hilbert space $H$, one may define an \textit{isonormal} or \textit{Gaussian white noise process} $\mathbb W_H$ which is centered with the covariance $\E [\mathbb W_H(h) \mathbb W_H(g)] = \langle h, g\rangle_H$ for all $h,g\in H$. Such a process always exists on a Hilbert space, as it can easily be proven that for any orthonormal basis $(e_j)_{j\in \mathbb N}\subset H$ and a sequence of i.i.d. standard normal $(\xi_j)_{j\in \mathbb N}$, the mapping 
\begin{equation*}
    h\mapsto \sum_{j=1}^\infty \xi_j \langle e_j, h\rangle
\end{equation*}
constitutes a white noise \parencite[Theorem 2.12]{da2014stochastic}.
This allows us to define generalized Gaussian white noise models on any Hilbert space $H$.
\begin{definition}[Generalized Gaussian white noise models]\label{def: general Gaussian White Noise}
    Let $H$ be a separable Hilbert space, $h\in H$ and $\sigma>0$. Let $\mathbb W_H$ be a white noise on $H$. We say that $\mathbb Y_h^{(n)}$ arises from a generalized Gaussian white noise model on $H$ and formally write 
    \begin{equation*}
        \mathbb Y_h^{(n)}  = h + \frac{\sigma}{\sqrt{n}}\mathbb W_{H}.
    \end{equation*}
\end{definition}

Observing the stochastic process $\mathbb Y_h^{(n)}$ simultaneously for the whole orthonormal basis $(e_j)_{j\in \mathbb N}$ of $H$ gives rise to the equivalent \textit{Gaussian Sequence Space Model} 
\begin{equation*}
    Y_k = \langle h, e_k\rangle_H + \frac{\sigma}{\sqrt{n}}\xi_k, \quad k\in \N
\end{equation*}
where $\xi_k = \mathbb W_H(e_k)\sim \mathcal N(0,\|e_k\|_H^2) = \mathcal N(0,1)$.

Since any separable Hilbert space $H$ is isometrically isomorphic to $\ell^2(\N)$, observing the generalized Gaussian white noise model on $H$ is always equivalent to observing the Gaussian sequence space model on $\ell^2(\N)$ in the sense that the Le Cam-distance between the two experiments is zero, see also \parencite[Theorem 1.2.1]{gine_mathematical_2021}. While the asymptotic equivalence (in the sense of Le Cam) between the random design and the white noise model was shown for the finite-dimensional case, it is not clear whether the same holds random design white noise model \eqref{eq: random design white noise} and the generalized Gaussian white noise model on $H = L^2(\gamma, \mathcal Y)$. 

The generalized Gaussian white noise model and the Gaussian sequence space model are especially convenient for likelihood calculations, since the Radon-Nikodym derivative between the appearing Gaussian measures is explicit by the \textit{Cameron-Martin formula}, see \parencite[Corollary 2.4.3]{bogachev1998gaussian} for the general result. For $\textcolor{red}{H=}\ell^2(\N)$, this result can also be directly deduced from Kakutani's theorem \parencite[Theorem 1]{kakutani1948equivalence}, see also \parencite[Proposition 6.1.4]{gine_mathematical_2021}.

Consequently, for $h, h' \in H$, denoting by $P_{h}^Y$ and $P_{h'}^Y$ the corresponding observation laws, it holds that
\begin{align}\label{eq: KL Gaussian}
    \operatorname{KL}(P_{h,n}^Y, P_{h',n}^Y) &= \frac{n}{2\sigma^2} \lVert h- h'\rVert_{H}^2.
\end{align}

\subsection{Minimax Lower Bounds for Holomorphic Functionals in the Generalized White Noise Model}\label{subsection: LB holomorphic functionals}

We now translate the metric entropy estimates of the Section \ref{section: metric entropy} into minimax lower bounds for $\mathcal H^{s}(L; \C)$ and $\mathcal H^{s,t}(L; \mathcal{Y}_{\C})$. In order to apply Fano's inequality for the $L^2$ or $L^\infty$ loss, we will need the corresponding metric entropy lower bounds in Proposition \ref{proposition: L2 metric entropy LB} and Proposition \ref{proposition: Linfty metric entropy LB}. The Yang-Barron method then requires us to control the Kullback-Leibler divergence of the corresponding observation laws. 
In the generalized white noise model for $H = L^2(\mathcal X, \mathcal Y, \gamma)$, this  is determined by the squared $L^2$-distance, which is why we will need the $L^2$- metric entropy upper bounds of Proposition \ref{proposition: L2 metric entropy UB}. This is summarized in the following Lemma, which is based on the results by \textcite{yang1999information}. 

\begin{lemma}[Minimax lower bounds based on metric entropy]\label{lemma: minimax metric entropy}
Let $\mathcal M\subset H$, where $H$ is a separable Hilbert space, and suppose that the observation laws satisfy 
\begin{align}\label{eq: KL-identity}
   KL(\P_f,\P_g) = \frac{n}{2\sigma^2}\|f-g\|_H^2, \qquad f,g\in H. 
\end{align}
Let $d$ be a semi-metric on $\mathcal M$. Suppose that, for some $\alpha,\beta>0$ and all sufficiently small $\varepsilon,\delta>0$, 
\begin{equation*}
    \log N(\varepsilon,\mathcal M,\|\cdot\|_H) \lesssim \varepsilon^{-\alpha}
\end{equation*}
and 
\begin{equation*}
    \log N(\delta,\mathcal M,d) \gtrsim \delta^{-\beta}. 
\end{equation*}
Then, for all sufficiently large $n$, 
\begin{equation*}
    \inf_{\widehat f} \sup_{f\in\mathcal M} \mathbb E_f[d(\widehat f,f)] \gtrsim \left(\frac{n}{\sigma^2}\right)^{ -\frac{\alpha}{\beta(\alpha+2)} }.
\end{equation*}
\end{lemma}
\begin{proof}
    Let $J$ be uniformly distributed over a finite $2\delta$-packing of $\mathcal M$ in the metric $d$. By Fano's inequality, 
    \begin{equation*}
        \inf_{\widehat f} \sup_{f\in\mathcal M} \mathbb E_f[d(\widehat f,f)] \geq \delta \left( 1-\frac{I(Z;J)+\log 2}{\log M(2\delta,\mathcal M,d)} \right).
    \end{equation*}
    By the Yang--Barron method \parencite{yang1999information} in the form of \eqref{eq: yang-barron}, and \eqref{eq: KL-identity}, 
    \begin{align*} 
        I(Z;J) 
        &\leq \inf_{\varepsilon>0} \left\{ \varepsilon^2+ \log N\!\left( \frac{\sqrt{2}\sigma \varepsilon}{\sqrt n}, \mathcal M, \|\cdot\|_H \right) \right\} \\ 
        &\lesssim \inf_{u>0} \left\{ u^2+ \left( \frac{\sqrt n}{\sigma u} \right)^\alpha \right\} \\
        &\lesssim \left( \frac{n}{\sigma^2} \right)^{\alpha/(\alpha+2)}. 
    \end{align*} 
    Choosing $\delta \asymp \left( \frac{n}{\sigma^2} \right)^{-\alpha/(\beta(\alpha+2))} $ with a sufficiently small implicit constant makes the logarithm of the packing number a sufficiently large multiple of the mutual information. Fano's inequality then gives the claimed bound.
\end{proof}

\begin{theorem}\label{theorem: lower bounds functional}
    Let $s>1$, $L>0$ and $\sigma>0$. Suppose that Assumption \ref{assumption: measure} holds. Set $\eta\coloneqq s- 1/2$ and $\xi \coloneqq s-1$. Then for every $\tau>0$, the minimax risk in the generalized white noise model on $H = L^2(\gamma)$ satisfies 
    \begin{equation*}
        \inf_{\widehat g_n} \sup_{g_0\in\mathcal H^s(L;\C)} \mathbb E_{g_0} \left[ \|\widehat g_n-g_0\|_{L^2(\gamma)} \right] \gtrsim n^{-\frac{\eta}{2\eta+1}-\tau} = n^{-\frac{s-1/2}{2s}-\tau}.
    \end{equation*}
    Moreover, 
    \begin{equation*}
        \inf_{\widehat g} \sup_{g_0\in\mathcal H^s(L;\C)} \mathbb E_{g_0} \left[ \|\widehat g-g_0\|_{\infty} \right] \gtrsim n^{-\frac{\xi}{2\eta+1}-\tau} = n^{-\frac{s-1}{2s}-\tau}.
    \end{equation*}
\end{theorem}
\begin{proof}
    By \eqref{eq: KL Gaussian}, the observation laws for the generalized Gaussian white noise model on $H = L^2(\gamma)$ satisfy the condition \eqref{eq: KL-identity} of Lemma \ref{lemma: minimax metric entropy}. Fix $\tau>0$ and let $\delta>0$ be arbitrarily small. By Proposition \ref{proposition: L2 metric entropy UB}, 
    \begin{equation*}
        \log N\!\left( \varepsilon,\mathcal H^s(L;\C),\|\cdot\|_{L^2(\gamma)} \right) \lesssim \varepsilon^{-\frac{1}{\eta}-\delta}.
    \end{equation*}
    For the $L^2$-loss, by Proposition \ref{proposition: L2 metric entropy LB}, we have 
    \begin{equation*}
        \log N\!\left( \varepsilon,\mathcal H^s(L;\C),\|\cdot\|_{L^2(\gamma)} \right) \gtrsim \varepsilon^{-\frac{1}{\eta}}.
    \end{equation*}
    By an application of Lemma \ref{lemma: minimax metric entropy} with $\alpha = 1/\eta + \delta$ and $\beta = 1/\eta$, the rate exponent is
    \begin{equation*}
        - \frac{\alpha}{\beta(\alpha + 2)} = - \frac{(1+ \delta \eta)\eta}{(1 + 2\eta + \delta \eta)} \geq - \frac{\eta}{1 + 2\eta} - \tau
    \end{equation*}
    for sufficiently small $\delta>0$.
    For the $L^\infty$-loss, Proposition~\ref{proposition: Linfty metric entropy LB} yields 
    \begin{equation*}
        \log N\!\left( \varepsilon,\mathcal H^s(L;\C),\|\cdot\|_\infty \right) \gtrsim \varepsilon^{-1/\xi}    
    \end{equation*}
    Applying Lemma~\ref{lemma: minimax metric entropy} with $\alpha=1/\eta+\delta$ and $\beta=1/\xi$ yields the second claim. 
\end{proof}

\subsection{Minimax Lower Bounds for Holomorphic Operators in the Generalized White Noise Setting}\label{subsection: GWN lower bounds}

The lower bounds for the class of holomorphic operators $\mathcal H^{s,t}(L; \mathcal Y_{\C})$ follow from analogous argument as in Section \ref{subsection: LB holomorphic functionals}. 

\begin{theorem}\label{theorem: Lower Bounds}
    Let $s>1$, $t>0$, $L>0$, and $\sigma>0$. Suppose that Assumption \ref{assumption: measure} holds. Set $\kappa_2\coloneqq\min\{s-1/2,t\}$ and $\kappa_\infty\coloneqq\min\{s-1,t\}$. Then, for any arbitrarily small $\tau>0$, 
    \begin{equation*}
        \inf_{\widehat G_n}\sup_{G_0 \in \mathcal H^{s,t}(L; \mathcal Y_{\C})} \mathbb E_{G_0}[\|\widehat G_n - G_0\|_{L^2(\gamma, \mathcal Y)}] \gtrsim n^{-\frac{\kappa_2}{2 \kappa_2 + 1} - \tau}.
    \end{equation*}
    Furthermore, 
    \begin{equation*}
        \inf_{\widehat G_n}\sup_{G_0 \in \mathcal H^{s,t}(L; \mathcal Y_{\C})} \mathbb E_{G_0}[\|\widehat G_n - G_0\|_{\infty, \mathcal Y}] \gtrsim n^{-\frac{\kappa_{\infty}}{2 \kappa_{2} + 1} - \tau}
    \end{equation*}
    for $\kappa_{\infty} = \min \{s-1, t\}$ and arbitrarily small $\tau>0$.
\end{theorem}
\begin{proof}
    By \eqref{eq: KL Gaussian}, the generalized white noise model on $H=L^2(\gamma,\mathcal Y)$ satisfties the condition \eqref{eq: KL-identity} of Lemma \ref{lemma: minimax metric entropy}.
    Fix $\delta>0$ arbitrarily small. By Proposition \ref{proposition: L2 metric entropy UB}, 
    \begin{equation*}
        \log N\!\left( \varepsilon, \mathcal H^{s,t}(L;\mathcal Y_\C), \|\cdot\|_{L^2(\gamma,\mathcal Y)} \right) \lesssim \varepsilon^{-\frac{1}{\kappa_2}-\delta}.
    \end{equation*}
    For the $L^2$-loss, Proposition \ref{proposition: L2 metric entropy LB} gives a metric entropy lower bound of order $\varepsilon^{-1/\kappa_2}$. Applying Lemma~\ref{lemma: minimax metric entropy} with $\alpha=1/\kappa_2+\delta$ and $\beta=1/\kappa_2$ yields the lower bound $n^{-\kappa_2/(2\kappa_2+1)-\tau}$.  For the $L^\infty$-loss, Proposition~\ref{proposition: Linfty metric entropy LB} yields a metric entropy lower bound of order $\varepsilon^{-1/\kappa_\infty}$. Applying Lemma \ref{lemma: minimax metric entropy} again with $\alpha=1/\kappa_2+\delta$ and $\beta=1/\kappa_\infty$ yields $ n^{-\kappa_\infty/(2\kappa_2+1)-\tau}$.  Since $\delta>0$ may be chosen arbitrarily small, so may $\tau>0$.
\end{proof}

\subsection{Minimax Lower Bounds in the Random Design White Noise Setting}\label{subsection: RWN lower bounds}

Because for any $G, G' \in L^2(\gamma, \mathcal Y)$ the Kullback–Leibler divergence between the corresponding observation laws in the random design white noise model \eqref{eq: random design white noise} coincides with that in the generalized Gaussian white noise model \eqref{eq: white noise}, the lower bounds established in Theorem \ref{theorem: lower bounds functional} and Theorem \ref{theorem: Lower Bounds} also hold for the random design white noise model. The full proof can be found in Appendix \ref{appendix: proof of lower bounds random design}.

\begin{theorem}\label{theorem: lower bounds random design}
    Assume that the observations arise from the random design white noise model \eqref{eq: random design white noise}. Let $s>1$, $t>0$ and $L>0$, and let Assumption \ref{assumption: measure} hold. Then the same lower bound for the minimax risk over $\mathcal H^{s,t}(L; \mathcal Y_{\C})$ hold with respect to the $L^2(\gamma, \mathcal Y)$-loss and the $L^\infty(\operatorname{supp}(\gamma),\mathcal Y)$-loss as in Theorem \ref{theorem: Lower Bounds}.
    For $\mathcal Y = \R$, the corresponding minimax lower bounds of Theorem \ref{theorem: lower bounds functional} for $\mathcal H^s(L;\C)$ hold with respect to the $L^2(\gamma)$ and the $L^\infty(\operatorname{supp}(\gamma))$-loss.
\end{theorem}

\section{Upper Bounds for Projection Estimators}\label{section: upper bounds}

The approximation rates derived in Section ~\ref{subsection: approximation rate} yield corresponding algebraic rates for the statistical estimation problem. In this section, we construct spectral projection estimators in the generalized Gaussian white noise model based on truncations of generalized polynomial chaos expansion. We prove that they achieve the same rates as those which were obtained by empirical risk minimizers in \parencite{reinhardt2026statistical}, nearly matching the lower bounds in $L^2$-risk.

\begin{definition}
    Let $H$ be a separable Hilbert space, $I$ some countable index set and let $(h_k)_{k\in I} \subset H$ be a Riesz basis of $H$ with lower and upper Riesz constants $d_H, D_H$ and dual basis $(\widetilde h_k)_{k\in I} \subset H$. Assume that $\mathbb Y_h^{(n)}$ arises from a generalized Gaussian white noise model on $H$ in the sense of Definition \ref{def: general Gaussian White Noise}. For any $N \in \N$ and any index set $\Lambda_N \subset I$, define the projection 
    \begin{equation*}
        h_N = \sum_{i\in \Lambda_N} \langle h, \widetilde h_i\rangle_H h_i
    \end{equation*}
    and the projection estimator 
    \begin{equation*}
        \widehat h_N  \coloneqq \sum_{i\in \Lambda_N} \mathbb Y_h^{(n)}(\widetilde h_i) h_i
    \end{equation*}
    for all $h\in H$, and if the dependence of $N\in \N$ on $n\in \N$ is clear, we write $\widehat h_n = \widehat h_{N(n)}$.
\end{definition}

Given uniform, quantitative rates for the approximation error $\|h - h_N\|$, this yields the following upper bound for the estimation error. 

\begin{theorem}[Upper Bounds for Projection Estimators]\label{theorem: Upper Bounds Projection Estimators White Noise}
    Let $\mathbb Y_h^{(n)}$ arise from a generalized Gaussian white noise model on a separable Hilbert space $H$, and let $(h_i)_{i\in \N}\subset H$ be a Riesz basis of $H$ with lower and upper Riesz constants $d_H, D_H$ and dual Riesz basis $(\widetilde h_j)_{j\in \N}$. Let $\mathcal J\subset H$. Assume that for every $N \geq 2$, there exists an index set $\Lambda_N$ of cardinality $N$ (independent of $h\in \mathcal J$), such that
    \begin{align}\label{eq: tail bound}
        \sup_{h\in \mathcal J}\left\|h - h_N\right\|_H^2 \leq C N^{-\xi}\log(N)^{\zeta}
    \end{align}
    for some $C>0$, $\xi >0$ and $\zeta \in \R$.
    Then the projection estimator $\widehat h_N$ satisfies 
    \begin{equation*}
        \sup_{h\in \mathcal J}\mathbb E_{h}\left[\left \|h-\widehat h_N\right\|^2_{H}\right] \leq 
            CN^{-\xi}\log(N)^{\zeta}  +\frac{D_H^2}{d_H^2}\frac{\sigma^2}{n}N.
    \end{equation*}
    Moreover, choosing 
    \begin{equation*}
        N \simeq n^{1/(\xi + 1)} \log(n)^{\zeta/(\xi + 1)} 
    \end{equation*}
    yields
    \begin{equation}\label{eq: minimax rate general projection}
            \sup_{h\in \mathcal J} \mathbb E_{h}\left[\left \|h-\widehat h_n\right\|^2_{H}\right] \lesssim n^{-\xi/(\xi + 1)}\log(n)^{\zeta/(\xi + 1)}.
    \end{equation}
\end{theorem}
\begin{proof}
    Since $\mathbb E[\widehat h_N] = h_N$, it holds that
    \begin{align}\label{eq: bias-variance}
        \mathbb E_h\left[\left \| \widehat h_N - h\right\|_H^2\right] &= \left\|h - \mathbb E[\widehat h_N]\right\|_H^2 + \mathbb E_h\left[\|\widehat h_N - \mathbb E_h [\widehat h_N]\|_H^2\right] \\
        &=  \left\|h -  h_N\right\|_H^2 + \frac{\sigma^2}{n} \mathbb E_h\left[\left\|\sum_{i\in \Lambda_N} \mathbb W_H(\widetilde h_i)h_i\right\|_H^2\right].\nonumber
    \end{align}
    By the upper Riesz bound for $(h_i)_{i\in \N}$, 
    \begin{equation*}
        \left \|\sum_{i\in \Lambda_N} \mathbb W_H(\widetilde h_i)h_i\right\|^2_H \leq D_H^2 \sum_{i\in \Lambda_N}(\mathbb W_H(\widetilde h_i))^2,
    \end{equation*}
    and furthermore, since $1/d_H$ is an upper Riesz bound for $(\widetilde h_i)_{i\in \N}$,
    \begin{equation*}
        \mathbb E [\mathbb W_H(\widetilde h_i)^2] = \|\widetilde h_i\|_H^2 \leq \frac{1}{d_H^2}.
    \end{equation*}
    Therefore, 
    \begin{align}\label{eq: stochastic term}
        \mathbb E_h \left[\left \|\sum_{i\in \Lambda_N} \mathbb W_H(\widetilde h_i)h_i\right\|^2_H\right]  \leq \frac{D_H^2}{d_H^2} N.
    \end{align}
If $(h_i)_{i\in \N}$ is an orthonormal basis, $d_H = D_H = 1$ and the inequality is an equality. Combining this estimate with the approximation bound \eqref{eq: tail bound} proves the first claim. Balancing both terms in \eqref{eq: bias-variance} gives $N\asymp n^{1/(\xi + 1)}\log(n)^{\zeta/(\xi + 1)}$ and yields \eqref{eq: minimax rate general projection} after substitution back into \eqref{eq: bias-variance}. 
\end{proof}

We first consider the scalar-valued class of holomorphic functionals $\mathcal H^{s}(L; \C)$. Let $n\in \N$ be fixed, let $N(n)\in \N$ and $\Lambda_N\subset \mathcal F$ be a finite set of multi-indices. Corresponding to the (noise-free) projection $g_N$ (see \eqref{eq: Legendre projection}
), define the projection estimator
\begin{equation}\label{eq: L2 proj est functional}
    \widehat g_n\coloneqq\sum_{\nu\in\Lambda_N}\mathbb Y_{g_0}^{(n)}\left(L_\nu\circ\sigma_s^{-1}\right)L_\nu\circ\sigma_s^{-1} = g_N + \frac{\sigma}{\sqrt{n}}\sum_{\nu \in \Lambda_N}\mathbb W_{L^2(\gamma)}(L_\nu\circ\sigma_s^{-1})L_\nu\circ\sigma_s^{-1}.
\end{equation}
Note that, by definition of the white noise, $\mathbb W_{L^2(\gamma)}(L_\nu \circ \sigma_s^{-1}) \sim N(0, 1)$, since $\|L_\nu\circ \sigma_s^{-1}\|^2 _{L^2(\gamma)} = \|L_\nu\|^2_{L^2(Q, \pi)}  = 1$.

\begin{theorem}[Statistical upper bounds for holomorphic functionals]\label{theorem: statistical upper bounds functional}
  Let $s>1$, $L>0$, $\sigma>0$ and $g_0\in \mathcal H^{s}(L; \C)$. Let $\gamma$ be as in Assumption \ref{assumption: measure}.
  Assume that $\mathbb Y_{g_0}^{(n)}$ arises from the generalized Gaussian white noise model on $L^2(\gamma)$. Set $\eta = s-1/2$ and $\xi = s-1$. Then, for every $\delta>0$ and $N\in \N$, there exists a downward closed index set $\Lambda_N$, independent of $g_0$, and a constant $C_2>0$ (depending only on $\delta$, $s$ and $L$) such that
    \begin{equation}\label{eq: L2 bias variance}
        \sup_{g_0 \in \mathcal H^{s}(L; \C)}\mathbb E_{g_0} [\lVert \widehat g_N - g_0\rVert_{L^2(\gamma)}^2]\leq C_2 N^{-2\eta + \delta} + \frac{\sigma^2}{n}  N.
    \end{equation}
    In particular, choosing $N \simeq n^{1/(2\eta - \delta + 1)}$, we obtain, for every $\tau>0$, 
    \begin{equation*}
        \sup_{g_0 \in \mathcal H^{s}(L; \C)}\mathbb E_{g_0} [\lVert \widehat g_n - g_0\rVert_{L^2(\gamma)}^2]\lesssim  n^{-\frac{2\eta}{2\eta+1} + \tau}.
    \end{equation*}
Similarly, for every $\delta>0$, there exists a downward closed index set $\Lambda_N$ and a constant $C_\infty>0$, such that 
    \begin{equation}\label{eq: Linfty bias variance}
        \sup_{g_0 \in \mathcal H^{s}(L; \C)}\mathbb E_{g_0} [\lVert \widehat g_n - g_0\rVert_{\infty}]\leq C_\infty N^{-\xi + \delta} + \frac{\sqrt 2\sigma}{\sqrt{n\pi}}  N^{3/2}.
    \end{equation}
Choosing $N \simeq n^{1/(2\xi - 2\delta + 3)}$ yields, for every $\tau>0$,
    \begin{equation*}
        \sup_{g_0 \in \mathcal H^{s}(L; \C)}\mathbb E_{g_0} [\lVert \widehat g_n - g_0\rVert_{\infty}]\lesssim  n^{-\frac{\xi}{2\xi+3} + \tau} = n^{-\frac{s-1}{2s+1} + \tau}.
    \end{equation*}
\end{theorem}
\begin{proof}
    For the $L^2$-bound \eqref{eq: L2 bias variance}
, Theorem \ref{theorem: realvalued holomorphic approximation} shows that there exists a downward closed $\Lambda_N$ and a constant $C_2>0$, independent of $g_0$, such that 
\begin{equation*}
    \sup_{g_0 \in \mathcal H^s(L;\C)} \|g_0 - g_N\|_{L^2(\gamma)}^2 \leq C_2 N^{-2\eta + \delta}.
\end{equation*}
Since $(L_\nu\circ \sigma_s^{-1})$ is an orthonormal basis of $H = L^2(\gamma)$, we may apply Theorem \ref{theorem: Upper Bounds Projection Estimators White Noise} with $d_H = D_H = 1$ to the projection estimator $\widehat g_N$, which yields \eqref{eq: L2 bias variance}. The two terms are balanced for $N \asymp n^{1/(2\eta - \delta + 1)}$. Choosing $\delta>0$ sufficiently small and substituting back into \eqref{eq: L2 bias variance} yields the claimed upper bound on the $L^2$-rate.

For the $L^\infty$-bound, Theorem \ref{theorem: realvalued holomorphic approximation} shows that there exists a downward closed set $\Lambda_N$ and a constant $C_\infty>0$, such that
\begin{equation*}
    \sup_{g_0 \in \mathcal H^s(L;\C)} \|g_0 - g_N\|_{\infty} \leq C_\infty N^{-\xi + \delta}.
\end{equation*}
By the triangle inequality, 
    \begin{align*}
        \mathbb E_{g_0}\left[\|\widehat g_N - g_0\|_\infty\right] \leq \|g_N - g_0\|_\infty + \mathbb E_{g_0}\left[\|\hat g_N - g_N\|_\infty\right].
    \end{align*}
    Let $\xi_\nu\coloneqq \mathbb W_{L^2(\gamma)}(L_\nu \circ \sigma_s^{-1})$. Since $(L_\nu \circ \sigma_s^{-1})_{\nu \in \mathcal F}$ is orthonormal in $L^2(\gamma)$, the $\xi_\nu$ are i.i.d. standard normal variables. The second term can be controlled by
    \begin{align*}
        E_{g_0}\left[\|\hat g_N - g_N\|_\infty\right]
        &\leq\frac{\sigma}{\sqrt{n}}\mathbb E_{G_0}\left[ \left\|\sum_{\nu \in \Lambda_{N}}\xi_\nu (L_\nu\circ \sigma_s^{-1} )\right\|_{_{\infty}}\right] \\
        &\leq \frac{\sigma}{\sqrt{n}} \sum_{\nu \in  \Lambda_{N}} \|L_\nu\|_{\infty}\mathbb E[|\xi_\nu|] \\
        &\leq \frac{\sqrt 2 \sigma}{\sqrt{n\pi}} \sum_{\nu \in \Lambda_N} w_\nu^{1/2},
    \end{align*}
    where we used the triangle inequality and that $\mathbb E{|\xi_\nu|} = \sqrt{2/\pi}$ . Lastly, by the Cauchy-Schwarz inequality and \parencite[Lemma 3.1]{chkifa2015discrete}, 
    \begin{equation*}
        \sum_{\nu \in \Lambda_N} w_\nu^{1/2} \leq |\Lambda_N|^{1/2} \left( \sum_{\nu\in\Lambda_N}w_\nu \right)^{1/2} \leq|\Lambda_N|^{3/2} = N^{3/2},
    \end{equation*}
    which proves \eqref{eq: Linfty bias variance}. Both terms are balanced for $N\asymp n^{1/(2\xi -2\delta + 3)}$ and choosing $\delta>0$ sufficiently small yields the claimed bound on the $L^\infty$-rate.
\end{proof}

Next, we consider the class of $\mathcal Y$-valued, holomorphic operators $\mathcal H^{s,t}(L;\mathcal Y_{\C})$. Let $\Phi = (\phi_j)_{j\in \N}$ be a Riesz basis of $\mathcal Y$ with dual basis $\widetilde\Phi = (\widetilde\phi_j)_{j\in \N}$. Let $n\in \N$ be fixed, let $N(n)\in \N$ and let $(\Lambda_{N,j})_{j\leq N}, (\Lambda^\infty_{N,j})_{j\leq N} \subset \mathcal F$ be two families of finite sets of multi-indices. For $x\in \sigma_s(Q)$, define 
\begin{equation}\label{eq: L2 proj est operator}
    \widehat {G}_{N} (x)\coloneqq \sum_{j=1}^\infty\sum_{\nu \in \Lambda_{N,j}} \mathbb Y_{G_0}^{(n)}([L_\nu\circ\sigma_s^{-1} ]\otimes \widetilde \phi_j) L_\nu(\sigma_s^{-1}(x))\phi_j
\end{equation}
and 
\begin{equation}\label{eq: Linfty proj est operator}
    \widehat {G}_{N}^\infty (x)\coloneqq \sum_{j=1}^\infty\sum_{\nu \in  \Lambda_{N,j}^{\infty}} \mathbb Y_{G_0}^{(n)}( [L_\nu\circ\sigma_s^{-1}] \otimes \widetilde \phi_j) L_\nu(\sigma_s^{-1}(x)) \phi_j.
\end{equation}
If the dependence of $N\in \N$ on $n\in \N$ is clear, we write $\widehat G_n = \widehat G_{N(n)}$ and $\widehat G_{n}^\infty = \widehat G_{N(n)}^\infty$.

\begin{theorem}\label{theorem: statistical upper bounds operator}
    Let $s>1$, $t>0$, $L>0$ and $\sigma>0$, let $G_0 \in \mathcal H^{s,t}(L; \mathcal Y_{\C})$ and let $\gamma$ be as in Assumption \ref{assumption: measure}. Assume that $\mathbb Y_{G_0}^{(n)}$ arises from the generalized Gaussian white noise model on $L^2(\gamma, \mathcal Y)$ and set $\kappa_2 = \min \{s-1/2, t\}$ and $\kappa_\infty = \min \{s-1,t\}$. Then, for every $\delta>0$ and $N\in \N$, there exist a constant $C_2>0$ and  downward closed sets $(\Lambda_{N,j})_{j\in \N}$, independent of $G_0$ and satisfying $\sum_{j=1}^\infty|\Lambda_{N,j}|\leq N$, such that 
    \begin{equation}\label{eq: L2 bias variance operator}
        \sup_{G_0 \in \mathcal H^{s,t}(L; \mathcal Y_{\C})}\mathbb E_{G_0}\left[\lVert \widehat {G}_N - G_0\rVert^2_{L^2(\gamma, \mathcal Y)}\right]\leq  
            C_2N^{-2\kappa_2 + \delta} + \frac{D_{\Phi}^2}{d_\Phi^2}\frac{\sigma^2}{n}N.
    \end{equation}
Choosing $N \simeq n^{1/(2\kappa_2 - \delta + 1)}$, we obtain for every $\tau>0$, that
    \begin{equation*}
        \sup_{G_0 \in \mathcal H^{s,t}(L; \mathcal Y_{\C})}\mathbb E_{G_0} [\lVert \widehat G_n - G_0\rVert_{L^2(\gamma, \mathcal Y)}^2]\lesssim  n^{-\frac{2\kappa_2}{2\kappa_2+1} + \tau}.
    \end{equation*}
Similarly, there exist a constant $C_\infty>0$ and downward closed sets $( \Lambda_{N,j}^\infty)_{j\in \N}$, independent of $G_0$, satisfying $\sum_{j=1}^\infty| \Lambda_{N,j}^\infty|\leq N$, such that 
    \begin{equation}\label{eq: Linfty bias variance operator}
        \mathbb E_{G_0}\left[\lVert \widehat {G}_N^\infty - G_0\rVert_{\infty, \mathcal Y}\right]\lesssim  
            C_\infty N^{-\kappa_{\infty} + \delta} + \frac{D_\Phi}{d_\Phi}\frac{\sigma}{\sqrt{n}}N^{3/2}.
    \end{equation}
In particular, choosing $N \simeq n^{1/(2\kappa_{\infty} - 2\delta + 3)}$ yields for every $\tau>0$, that
    \begin{equation*}
        \sup_{G_0 \in \mathcal H^{s,t}(L; \mathcal Y_{\C})}\mathbb E_{G_0} [\lVert \widehat G_n^\infty - G_0\rVert_{\infty, \mathcal Y}]\lesssim  n^{-\frac{\kappa_{\infty}}{2\kappa_{\infty}+3} + \tau}.
    \end{equation*}
\end{theorem}
\begin{proof}
Since $(L_\nu\circ \sigma_s^{-1})_{\nu \in \mathcal F}$ is an orthonormal basis of $L^2(\gamma)$ and $(\phi_j)_{j\in \N}$ is a Riesz basis of $\mathcal Y$, $((L_\nu\circ \sigma_s^{-1}) \otimes \phi_j)$ is a Riesz basis of $H = L^2(\gamma, \mathcal Y) \simeq L^2(\gamma)\otimes \mathcal Y$,  with lower and upper Riesz constant $d_\Phi, D_\Phi$.
By Theorem \ref{theorem: operator approximation rate}, there exists a constant $C_2>0$ and downward closed sets $(\Lambda_{N,j})_{j\leq N}$ which satisfy $\sum_{j=1}^\infty |\Lambda_{N,j}|\leq N$, and such that
\begin{equation*}
    \sup_{G_0 \in \mathcal H^{s,t}(L; \mathcal Y_{\C})} \| G_N - G_0 \|_{L^2(\gamma, \mathcal Y)}^2 \leq C_2 N^{-2\kappa_2 + \delta}.
\end{equation*}
Applying Theorem \ref{theorem: Upper Bounds Projection Estimators White Noise} to the projection estimator $\widehat G_N$ proves \eqref{eq: L2 bias variance operator}. Both terms are balanced for  $N \simeq n^{1/(2\kappa_{\infty} - 2\delta + 3)}$ and choosing $\delta>0$ sufficiently small yields the claimed $L^2$-upper bound the statistical convergence rate.

For the $L^\infty$-bound, Theorem \ref{theorem: operator approximation rate} implies that there exists a constant $C_\infty>0$ and downward closed sets $(\Lambda_{N,j}^\infty)_{j\leq N}$ which satisfy $\sum_{j=1}^\infty |\Lambda_{N,j}^\infty|\leq N$, and such that
\begin{equation*}
    \sup_{G_0 \in \mathcal H^{s,t}(L; \mathcal Y_{\C})} \| \widetilde G_N - G_0 \|_{\infty, \mathcal Y} \leq C_\infty N^{-\kappa_\infty + \delta}.
\end{equation*}
Define 
\begin{equation*}
    Z_{\nu,j} = \mathbb W_{L^2(\gamma, \mathcal Y)}((L_\nu\circ \sigma_s^{-1})\otimes \widetilde\phi_j),
\end{equation*}
which by definition of Gaussian white noise has distribution $Z_{\nu,j}\sim N(0, \|\widetilde\phi_j\|_{\mathcal Y}^2)$ and thus $\mathbb E[|Z_{\nu,j}|] = \sqrt{2/\pi} \|\widetilde\phi_j\|_{\mathcal Y}$. By the triangle inequality, 
\begin{equation*}
    \mathbb E_{G_0}[\|\hat G_N^\infty - G_0\| _{\infty, \mathcal Y}] \leq \|\widetilde G_N - G_0\|_{\infty, \mathcal Y} + \mathbb E_{G_0}[\|\widehat G_N^\infty - \widetilde G_N\|_{\infty, \mathcal Y}].
\end{equation*}
and by \parencite[Lemma 3.1]{chkifa2015discrete} and again the triangle inequality
\begin{align*}
    \mathbb E[\|\widehat G_N^\infty - \widetilde G_N\|_{\infty, \mathcal Y}]&= \frac{\sigma}{\sqrt{n}}\mathbb E_{G_0}\left[ \left\|\sum_{j=1}^\infty \sum_{\nu \in \Lambda_{N,j}^\infty}Z_{\nu,j}(L_\nu\circ \sigma_s^{-1}) \otimes  \phi_j\right\|_{\infty, \mathcal Y}\right] \\
    &\leq \frac{\sqrt 2\sigma}{\sqrt{n\pi}}\sum_{j=1}^\infty \sum_{\nu \in \Lambda_{N,j}^\infty} \|L_\nu\|_{\infty}\|\widetilde \phi_j\|_{\mathcal Y}\|\phi_j\|_{\mathcal Y} \\
    &\leq \frac{\sqrt 2\sigma}{\sqrt{n\pi}} \frac{D_\Phi}{d_\Phi}\sum_{j=1}^\infty \sum_{\nu \in \Lambda_{N,j}^\infty} w_\nu^{1/2}\\
    &\leq \frac{\sqrt 2\sigma}{\sqrt{n\pi}} \frac{D_\Phi}{d_\Phi}   \sum_{j=1}^\infty |\Lambda_{N,j}^\infty|^{3/2}\\
    &\leq  \frac{\sqrt 2\sigma}{\sqrt{n\pi}} \frac{D_\Phi}{d_\Phi} (\sum_{j=1}^\infty |\Lambda_{N,j}^\infty|)^{3/2}\\
    &\leq  \frac{\sqrt 2\sigma}{\sqrt{n\pi}} \frac{D_\Phi}{d_\Phi} N^{3/2}.
\end{align*}
Balancing both terms and choosing $\delta>0$ sufficiently small completes the proof.
\end{proof}

\begin{remark}
    The same $L^2$- upper bounds were obtained in the random design white noise model \eqref{eq: random design white noise} in \parencite[Theorem 3.15]{reinhardt2026statistical} for empirical risk minimizers over classes of neural operator architectures with ReLU activations, called \textit{Sparse FrameNets}. As shown in Section \ref{section: lower bounds}, these $L^2$ upper bounds are matched (up to an arbitrarily small loss in the exponent) by the lower bound and the estimators thus nearly achieve the minimax-optimal rate over the class of operators $\mathcal H^{s,t}(L; \mathcal Y_{\C})$. Establishing sharp $L^\infty$-upper bounds would require sharper control of the stochastic error term arising from the noisy observations.
\end{remark}

\section{Application to Parametric Elliptic PDEs}\label{section: examples}

To give a concrete example, we recall and continue the discussion of the standard parametric elliptic PDE on the torus from \cite[Section 7.1]{herrmann_neural_2024}. Let $d\in\N$ and denote by $\mathbb{T}^d$ the $d$-dimensional torus. We consider the second order PDE
\begin{equation*}
  -\nabla\cdot(a\nabla u)=f\qquad\text{on}\qquad\mathbb{T}^d.
\end{equation*}
Here $f\in C^\infty(\mathbb{T}^d)$ with $\int_{\mathbb{T}^d}f=0$, and
$a\in L^\infty(\mathbb{T}^d)$ is a complex-valued diffusion coefficient
whose real part is uniformly bounded from below by a positive constant. Standard elliptic PDE theory implies the existence of a unique solution $u\in H^1(\mathbb{T}^d)$ such that $\int_{\mathbb{T}^d}u=0$, which we denote in the following by $G_0(a)$, i.e.\ $G_0:a\mapsto u$ is the solution operator. If we assume additional regularity on $f$, such as $f\in C^\infty$, then a shift theorem holds, and higher regularity in $a$ transfers to higher regularity of the solution $G(a)$, as we explain next.

With a suitable choice of (Fourier) basis on the torus, and fixing
$s_0$, $t_0\ge 0$, one can show that the function spaces in
\eqref{def: scale spaces} correspond to
\begin{equation*} \mathcal{X}:=H^{s_0}(\mathbb{T}^d),\qquad\mathcal{Y}:=H^{t_0}(\mathbb{T}^d),\qquad\mathcal{X}^s=H^{s_0+sd}(\mathbb{T}^d),\qquad\mathcal{Y}^t=H^{t_0+td}(\mathbb{T}^d).
\end{equation*}
Let $d\geq 2$, $\alpha > \frac{3d}{2}$ and $t_{0}\in [0,1]$.  Fix
$\tau>0$ and $\tau_1\in \left( 0, \alpha - 3d/2 \right)$, $\delta>0$
all arbitrarily small, and set
$$
    s_{0} = \begin{cases}
    \alpha - d - \tau_1& if\; \alpha \in (\frac{3d}{2}, 2d + t_{0} - 1]\\ 
    \frac{ \alpha + t_{0} - 1}{2} & if \; \alpha> 2d + t_{0} - 1,
            \end{cases}
            \qquad
            s = \frac{\alpha - s_0}{d}, \qquad t = \frac{1 + s_0  -d/2 - t_0}{d} - \delta.
$$
Well-known holomorphy of the
    solution operator
    \cite{cohen_convergence_2010,cohen_analytic_2011} together with
    classical PDE regularity theory \cite{MR162050} then allow to show
    \cite[Proposition~3, Step~1]{herrmann_neural_2024}
\begin{equation*}
  G_0\in \mathcal{H}^{s,t}(L;\mathcal{Y}_{\C}),
\end{equation*}
for some $L>0$ and suitable sets $O_{\C}$ and $\mathcal{D}_{\C}$
in Definition \ref{def: holomorphic operator}
(constraining the real part of the diffusion coefficients in
particular to remain uniformly positive). The design measure $\gamma$
is in the following as in Assumption~\ref{assumption: measure},
shifted by some uniformly positive $\bar a\in L^\infty(\mathbb{T}^d)$,
such that its support lies in $O_{\C}$. For complete definitions and
statements, see \cite[Section 7.1]{herrmann_neural_2024}.

Let us now explain our results in this context. 
\paragraph{$L^2$-estimation}:
First, in the generalized white noise model from Section \ref{subsection: white noise} and for the $L^2$-loss, according to Theorem~\ref{theorem: statistical upper bounds operator}, the projection estimator $\widehat G_n$ \eqref{eq: L2 proj est operator} from Section~\ref{section: upper bounds} satisfies
\begin{subequations}\label{eq:torusupper}
\begin{equation}
  \mathbb{E}_{G_0}\big[\|\widehat G_n-G_0\|_{L^2(\gamma,\mathcal{Y})}^2\big]\lesssim n^{-\frac{2\kappa_2}{2\kappa_2+1}+\tau}
\end{equation}
with
\begin{equation}
  \kappa_2=\min\{s-1/2,t\}\geq
  \begin{cases}
    \frac{\alpha+1-t_0}{d}-\frac{3}{2} & \text{if }\alpha\in(\frac{3d}{2}, 2d+t_0-1],\\
    \frac{\alpha+1-t_0}{2d}-\frac{1}{2} & \text{if }\alpha>2d+t_0-1,
  \end{cases}
  \;-\frac{\tau_1}{d}-\delta.
\end{equation}
\end{subequations}
Second, if we want to recover the PDE solution operator among all holomorphic operators in $\mathcal{H}^{s,t}(L;\mathcal{Y}_{\C})$, Theorem~\ref{theorem: Lower Bounds} gives the lower bound
\begin{equation*}
  \inf_{\widehat G_n}\sup_{G\in\mathcal{H}^{s,t}(L;\mathcal{Y}_{\C})}\mathbb{E}_{G}\big[\|\widehat G_n-G\|_{L^2(\gamma,\mathcal{Y})}^2\big]\gtrsim n^{-\frac{2\kappa_2}{2\kappa_2+1}-\tau}.
\end{equation*}
Importantly, this matches the upper bound \eqref{eq:torusupper} up to
(the arbitrarily small) $\tau>0$. We also note that this is the same
convergence rate obtained in
\cite[Theorem~3.15]{reinhardt2026statistical} for empirical risk
minimization over a class of neural network based
operators\footnote{The choice of $s_0$ in
  \cite[Theorem~4.2]{reinhardt2026statistical} was not optimal for the case $\alpha\le 2d+1-t_0$. The choice here leads to the same rate in \cite[Theorem~4.2]{reinhardt2026statistical}.}.  Hence, improving this rate
would require exploiting structure beyond holomorphy in the sense of
membership in $\mathcal{H}^{s,t}$ (see Definition \ref{def: holomorphic operator}).  

\paragraph{$L^\infty$-estimation:} Next, we consider estimation with respect to the $L^\infty$-loss, in which case the approximation exponent is $\kappa_\infty = \min\{s-1,t\}$, see \parencite[Proposition~3, Step~2]{herrmann_neural_2024}. To optimize the exponent, we choose $s_0$ differently here to maximize the convergence rate. More precisely, choose
\begin{equation}\label{eq: s0 Linfty}
    s_{0} = \begin{cases}
    \frac{d}{2}+ \tau_1& if\; \alpha \in (\frac{3d}{2}, 2d + t_{0} - 1]\\ 
    \frac{ \alpha + t_{0} - 1}{2} & if \; \alpha> 2d + t_{0} - 1,
    \end{cases}
    \qquad
    s = \frac{\alpha - s_0}{d}, \qquad t = \frac{1 + s_0  -d/2 - t_0}{d} - \delta
\end{equation}
as in \parencite[Proposition 3]{herrmann_neural_2024}.
The holomorphy result \parencite[Proposition~3, Step~1]{herrmann_neural_2024} again implies that 
\begin{align*}
    G_0 \in \mathcal H^{s,t}(L;\mathcal Y_{\C})
\end{align*}
for suitable choices of $L>0$, $\mathcal O_{\C}$ and $\mathcal D_{\C}$. In the generalized white noise model, Theorem~\ref{theorem: statistical upper bounds operator} yields, for the projection estimator $\widehat G_n^\infty$ from \eqref{eq: Linfty proj est operator},
\begin{equation*}
    \mathbb E_{G_0}\left[\|\widehat G_n^\infty - G_0\|_{\infty, \mathcal Y}\right]\lesssim n^{-\frac{\kappa_{\infty}}{2\kappa_{\infty} + 3} + \tau}
\end{equation*}
with 
\begin{equation*}
    \kappa_{\infty} \geq \begin{cases}
        \frac{\alpha}{d} - \frac{3}{2}& \text{if }\alpha \in \left(\frac{3d}{2}, \frac{3d}{2}+ 1 -t_0\right]\\
        \frac{\alpha + 1 -t_0}{2d} - \frac{3}{4}& \text{if } \alpha > \frac{3d}{2} + 1 -t_0
    \end{cases} - \frac{\tau_1}{d} - \delta.
\end{equation*}
In contrast to the $L^2$-case, the resulting upper bound does not in general match the information-theoretic lower bound over $\mathcal H^{s,t}(L; \mathcal Y_{\C})$. Writing 
\begin{equation*}
    \widetilde \kappa_2 = \min \{s-1/2,t\},
\end{equation*}
with the choice of $s_0, s$ and $t$ as in \eqref{eq: s0 Linfty}, Theorem~\ref{theorem: Lower Bounds} yields the lower bound 
\begin{equation*}
    \inf_{\hat G_n}\sup_{G \in \mathcal H^{s,t}(L;\mathcal Y_{\C})}\mathbb{E}_{G}\left[ \lVert \widehat{G}_{n} - G \rVert _{\infty, \mathcal Y}\right] \gtrsim n^{-\frac{\kappa_\infty}{2\widetilde\kappa_2+1}-\tau}
\end{equation*}
over $\mathcal H^{s,t}(L;\mathcal Y_{\C})$.
\paragraph{Scalar quantities-of-interest: }
Finally, we apply the same framework to scalar quantities-of-interest of the PDE solution. Let $O \in \mathcal Y'$ be a bounded, linear observation operator with $\|O\|_{\mathcal{Y'}} =1$ and define $g_0 = O\circ G_0$. Since $O$ is bounded and linear, $G_0 \in \mathcal H^{s,t}(L;\mathcal Y_{\C})$ implies $g_0 \in\mathcal H^{s}(L', \C)$ for some $L'>0$. Choose 
\begin{equation*}
    s_0 = \frac{d}{2} + \tau_1,\qquad s = \frac{\alpha - s_0}{d} = \frac{\alpha}{d}-\frac{1}{2} - \frac{\tau_1}{d} 
\end{equation*}
and set 
\begin{equation*}
    \eta = s-\frac{1}{2} = \frac{\alpha}{d}-1 - \frac{\tau_1}{d}, \qquad \xi = s-1 = \frac{\alpha}{d}- \frac{3}{2} - \frac{\tau_1}{d}.
\end{equation*}
In the generalized Gaussian white noise model on $L^2(\gamma)$, Theorem~\ref{theorem: statistical upper bounds functional} implies that the projection estimator \eqref{eq: L2 proj est functional} satisfies 
\begin{equation}\label{eq: functional torus upper}
    \mathbb E_{g_0}\|\widehat g_n-g_0\|_{L^2(\gamma)}\lesssim n^{-\frac{\eta}{2\eta+1}+\tau}.
\end{equation}
and if we want to recover $g_0$ among all holomorphic functionals $\mathcal H^{s}(L';\C)$, Theorem~\ref{theorem: lower bounds functional} yields the lower bound
\begin{equation*}
    \inf_{\widehat g_n}\sup_{g\in\mathcal H^{s}(L'; {\C})}\mathbb E_{g}\|\widehat g_n-g\|_{L^2(\gamma)} \gtrsim n^{-\frac{\eta}{2\eta+1}-\tau}.
\end{equation*}
Again, as for the Hilbert-space-valued case, this matches the upper bound \eqref{eq: functional torus upper} up to (an arbitrarily small) $\tau>0$. For the $L^\infty$-norm, Theorem \ref{theorem: statistical upper bounds functional} yields that the projection estimator $\widehat g_n$ satisfies
\begin{equation*}
    \mathbb E_{g_0}\|\widehat g_n-g_0\|_{\infty}\lesssim
        n^{-\frac{\xi}{2\xi+3}+\tau}
\end{equation*}
with the  lower bound 
\begin{equation*}
    \inf_{\widehat g_n}\sup_{g\in\mathcal H^{s}(L'; {\C})}\mathbb E_{g}\|\widehat g_n-g\|_{\infty}\gtrsim n^{-\frac{\xi}{2\eta+1}-\tau}
\end{equation*}
over $\mathcal H^{s}(L';\C)$.

\section{Conclusion and Further Directions}

In this paper, we derived information-theoretic lower bounds for the sample complexity for estimating holomorphic functionals $g_0: \mathcal X \to \R$ and operators $G_0: \mathcal X\to \mathcal Y$ under the \textit{random design white noise model} and the \textit{operator white noise model}. In the latter model, the projection-type estimators developed in this paper attain $L^2$-convergence rates that match, up to an arbitrarily small loss in the exponent, the corresponding lower bounds over a class of holomorphic operators. In the random-design setting, these algebraic convergence rates were previously established by \textcite{reinhardt2026statistical} for empirical risk minimizers over classes of Sparse FrameNets for a related class of holomorphic operators.

Therefore, our results establish algebraic minimax rates for the estimation of holomorphic operators in the presence of observation noise, which complements earlier lower bounds and optimal recovery results for operator learning of holomorphic maps without observation noise, as well as recent information-theoretic lower bounds for noisy operator learning of Lipschitz and Hölder operators.

In particular, this resolves, for the class of holomorphic operators, the open question highlighted by \textcite{brugiapaglia2026short}, whether the rates obtained in \parencite{reinhardt2026statistical} are minimax optimal.

We also demonstrated the applicability of our results by considering examples of coefficient-to-solution maps as well as coefficient-to-quantity-of-interest maps arising from linear elliptic partial differential equations. We note that the holomorphy assumption is well-established for a wide array of models including general elliptic PDEs \parencite{chkifa_breaking_2015, cohen_convergence_2010, cohen_analytic_2011}, parabolic PDEs \parencite{hoang2013sparse}, Maxwell's equations \parencite{jerez2017electromagnetic}, the Calderon projector \parencite{henriquez2021shape}, the Helmholtz equation \parencite{hiptmair2018large} and Navier-Stokes-equations \parencite{cohen2018shape}. Several questions remain open. It would be natural to extend the minimax theory of this paper to \textit{fixed-design regression}, since sparse-grid collocation methods are known to achieve comparable $L^\infty$- approximation rates, and to \textit{other $L^p$-type losses} for $p\neq 2$ and $p \neq \infty$. Another possible direction is to investigate whether the \textit{asymptotic equivalence in the sense of Le Cam} holds for the random design white noise model and the generalized Gaussian white noise model. Finally, proving lower bounds for holomorphic maps on Gaussian spaces would be of particular interest. In that setting, the relevant expansions are based on Wiener-Hermite polynomials, and the corresponding complex extension domains are complex strips. First results on the approximation theory of such operators have recently been obtained in \parencite{dung_analyticity_2023, marcati2026sparsity, marcati2026neural}.

\section*{Acknowledgements and AI disclosure}
LDH, SW and JZ gratefully acknowledge support from the German Research Foundation (DFG) within the Priority Programme SPP 2298, \textit{Theoretical Foundations of Deep Learning} (project number 543965776). LDH additionally acknowledges support by the German Research Foundation (DFG) under Germany's Excellence Strategy – The Berlin Mathematics Research Center MATH+ (EXC-2046/2, project ID: 390685689), as well as support by a doctoral scholarship from the German Academic Scholarship Foundation (Studienstiftung des deutschen Volkes).

During the preparation of this manuscript, ChatGPT 5.6 was used for proofreading. The authors independently verified all final mathematical arguments and reported results and take full responsibility for
the manuscript.

\appendix
\section{Proofs from Section \ref{section: holomorphy class}}\label{appendix: proof section 2}

\subsection{Compactness}\label{appendix: compactness}
\begin{proposition}\label{proposition: compactness}
    The class $\mathcal H^{s,t}(L; \mathcal Y_{\C})$ is relatively compact in $L^\infty (\operatorname{supp}(\gamma), \mathcal Y)$ and, consequently, in $L^2(\gamma, \mathcal Y)$.
\end{proposition}
\begin{proof}
    Since $\operatorname{supp}(\gamma) = \sigma_s(Q)$ and $\gamma$ is a probability measure, we can bound 
    \begin{equation*}
        \|G\|_{L^2(\gamma, \mathcal Y)} \leq \|G\|_{\infty, \mathcal Y}
    \end{equation*}
    for all $G\in L^\infty (\operatorname{supp}(\gamma), \mathcal Y)$. It thus suffices to prove relative compactness of $\mathcal H^{s,t}(L; \mathcal Y_{\C})$ in $C(\sigma_s(Q), \mathcal Y)$. By the Arzela-Ascoli Theorem \parencite[Theorem 47.1]{munkres_topology_2000}, we need to show that $\mathcal H^{s,t}(L;\mathcal Y_{\C})$ is equicontinuous and pointwise relatively compact, i.e.~for each $x\in \operatorname{supp}(\gamma)$, the $\mathcal Y$-closure of the set $\{G(x): G\in \mathcal H^{s,t}(L; \mathcal Y_{\C})\}$ is compact in $\mathcal Y$. 
    
    Note that by Definition \ref{def: holomorphic operator}, for any $x\in \sigma_s(Q)\subset \mathcal D_{\C}$, it holds that $\|G(x)\|_{\mathcal Y^t}\leq L$ and is thus contained in a ball of radius $L$ of $\mathcal Y^t$. Since the embedding $\mathcal Y^t\hookrightarrow \mathcal Y$ is compact, $\mathcal H^{s,t}(L; \mathcal Y_{\C})$ is pointwise relatively compact. To show that $\mathcal H^{s,t}(L; \mathcal Y_{\C})$ is equicontinuous, let $x_0 \in \sigma_s(Q)$ and $h\in \mathcal X_{\ell^1}$ and $r>0$ small enough such that $\|h\|_{\mathcal X_{\ell^1}}= 1$ and $B_r^{\mathcal X_{\ell^1}}(x_0)\subset \mathcal D_{\C}$. By Cauchy's integral formula, it holds that for any $\xi \in B_{r/2}^{\C}(0)$ and any $G\in \mathcal H^{s,t}(L; \mathcal Y_{\C})$, we have 
    \begin{align*}
        \|G(x_0 + \xi h) - G(x_0)\|_{\mathcal Y} &\leq  \frac{1}{2\pi}\int_{|\zeta|= r} \left\|G(x_0 + \zeta h)\right\|_{\mathcal Y}\left(\frac{1}{\zeta - \xi} - \frac{1}{\zeta}\right)\; d\zeta\\
        & \leq  \frac{1}{2\pi}2\pi r \sup_{\xi \in \partial B_r^{\C}(0)} \frac{\left\|G(x_0 + \zeta h)\right\|_{\mathcal Y^t}|\xi|}{|\zeta - \xi||\zeta|}\\
        & \leq \frac{L}{r}|\xi|= \frac{L}{r}\|\xi h\|_{\mathcal X_{\ell^1}},
    \end{align*}
    where we have bound the norm of the integral by the integral of the norm in the first line and used the embedding $\mathcal Y^t \to \mathcal Y$ in the second line.
    Since $L$ is independent of $G$ and $\xi$, the statement follows.
\end{proof}

\subsection{Proof of Lemma \ref{lemma: weighted seq space}}\label{appendix: proof of weighted seq space}

    For $c\in \ell^1_{-s}(\N)$, the upper Riesz bound implies that $\|\psi_j\|_{\mathcal X}\leq D_\Psi$ for every $j\in \N$. Therefore, 
    \begin{equation*}
        \left\|\sum_{j=1}^{\infty} c_j j^{-s} \psi_j\right\|_{\mathcal X} \leq  \sum_{j=1}^{\infty} |c_j| j^{-s}\|\psi_j\|_{\mathcal X} \leq D_\Psi \|c\|_{\ell^1_{-s}(\N)}.
    \end{equation*}
    and thus $\|\widetilde \sigma_s(c)\|_{\mathcal X} \leq D_\Psi \|c\|_{\ell^1_{-s}}$.
    By the biorthogonality of $\Psi$ and $ \widetilde \Psi$,
    \begin{equation*}
        \|\widetilde\sigma_s(c)\|_{\mathcal X_{\ell^1}} = \sum_{j=1}^{\infty} |\langle \widetilde\sigma_s(c), \widetilde \psi_j \rangle | =\sum_{j=1}^{\infty} |c_j| j^{-s} =  \|c\|_{\ell^1_{-s}(\N)}.
    \end{equation*}
    Conversely, if $x\in \mathcal X_{\ell^1}$, set $c_j = j^{s} \langle x, \widetilde \psi_j \rangle_{\mathcal X}$.
    Then 
    \begin{equation*}
        \|c\|_{\ell^1_{-s}} = \sum_{k=1}^{\infty}|\langle x,\widetilde\psi_k\rangle_{\mathcal X}|=\|x\|_{\mathcal X_{\ell^1}}<\infty,
    \end{equation*}
    and 
    \begin{equation*}
        \widetilde \sigma_s(c) = \sum_{k=1}^\infty \langle x, \widetilde \psi_k\rangle_{\mathcal X} \psi_k = x.
    \end{equation*}
    Thus, $\widetilde \sigma_s$ maps $\ell^1_{-s}$ bijectively onto $\mathcal X_{\ell^1}$. The complex case follows identically.

\subsection{Proof of Proposition \ref{proposition: extension to infdim ellipse}}\label{appendix: proof of extension}

Fix $s>1$ and $r>0$. Recall that a sequence $\boldsymbol{\varrho} = (\varrho_j )_{j\in \N}$ is called $r$-\textit{admissible}, if 
\begin{equation*}
    \sum_{j=1}^\infty (\varrho_j -1)j^{-s} < r
\end{equation*}
and that we define
\begin{equation*}
    S_{s,r} = \bigcup_{\boldsymbol{\varrho} \text{ is r-admissible}}\mathcal{E}_{\boldsymbol{\varrho}}^\infty.
\end{equation*}
For $z\in  \C$, define $\varrho^*(z)\coloneqq \inf\{\varrho>1: z \in \mathcal E_\varrho \}$ and set $\varrho^*(z) =1$ for $z\in [-1,1]$. Furthermore, set 
\begin{equation*}
    R^*: \ell^1_{-s}(\N) \to (0,\infty), \quad R^*(z) = \sum_{j=1}^\infty (\varrho^*(z_j)-1)j^{-s}.
\end{equation*}
\paragraph{$R^*$ is continuous.}
    The map $\varrho^*: \C \to [1, \infty)$ is continuous. Furthermore, for $z\not \in [-1,1]$, $\varrho^*(z) = |z + \sqrt{z^2-1}|>1$, which yields the bound $\varrho^*(z)-1\leq 2|z|$. Consequently, for every $z\in \ell^1_{-s}(\N)_{\C}$, it holds that 
    \begin{equation}\label{eq: weighted bound}
        R^*(z) \leq 2 \sum_{j=1}^\infty |z_j|j^{-s} = 2 \|z\|_{\ell^1_{-s}},
    \end{equation}
    and $R$ is well-defined.
    
    Let $\varepsilon>0$ and fix $z\in \ell^1_{-s}$. Choose $N\in \N$ large enough such that 
    \begin{equation*}
        \sum_{j>N} |z_j|j^{-s} < \varepsilon/16.
    \end{equation*}
    Set $M = 1+ \max_{j\in \{1, \dots, N\}}|z_j|$. Since $\varrho^*$ is continuous, it is uniformly continuous on the compact disk $\overline B_M^{\C}(0)$. Therefore, there exists a $\eta>0$, such that $|z' -  z"| < \eta$ implies that 
    \begin{equation*}
        |\varrho^*(z') - \varrho^*(z")| < \frac{\varepsilon}{2 \sum_{j=1}^N j^{-s}}
    \end{equation*}
    for all $z', z" \in \overline B_M^{\C}(0)$. Now choose $\delta = \min \{\varepsilon/8,\eta N^{-s}, N^{-s}\}$. If  $\|\widetilde z - z\|_{\ell^1_{-s}}< \delta$, then for all $j \in \{1, \dots, N\}$, it holds that 
    \begin{equation*}
        |z_j - \tilde z_j| \leq j^{s} \|z - \tilde z\|_{\ell^1_{-s}} \leq N^s \delta  \leq \min \{1,\eta\}.
    \end{equation*}
    By an application of the triangle inequality, it holds that 
    \begin{equation*}
        |R^*(z) - R^*(\widetilde z)| \leq \sum_{j=1}^N |\varrho^*(z_j) -\varrho^*(\widetilde z_j)|j^{-s} + \sum_{j>N} |\varrho^*(z_j) -\varrho^*(\widetilde z_j)|j^{-s}
    \end{equation*}
    To upper bound the first term, note that $\tilde z_j \in \overline B_M^{\C}(0)$  and therefore 
    \begin{equation*}
         \sum_{j=1}^N |\varrho^*(z_j) -\varrho^*(\widetilde z_j)|j^{-s}  < \frac{\varepsilon}{2}.
    \end{equation*}
    For the second term, note that
    \begin{align*}
    \sum_{j=N+1}^\infty |\varrho^*(z_j) -\varrho^*(\widetilde z_j)|j^{-s}   &\leq 2\sum_{j=N+1}^\infty (|z_j| + |\widetilde z_j|) j^{-s}  \\
        &\leq 4 \sum_{j=N+1}^\infty |z_j|j^{-s} + 2\|z-\widetilde z\|_{\ell_{-s}^1} \\
        &= \frac{\varepsilon}{2}
    \end{align*}
    The continuity of $R$ follows.
    
    \paragraph{Characterizing $S_{s,r}$ by $\{z\in \ell^1_{-s}(\N): R^*(z) < r\}$.} 
    If $z\in S_{s,r}$, there exists an $r$-admissible sequence $\boldsymbol{\varrho} = (\varrho_j)_{j\in \N}$, such that $z\in \mathcal{E}_{\boldsymbol{\varrho}}^\infty$. Then there exists a $(c_j)_{j\in \N}\in [-1,1]^{\N}$, such that $|z_j - c_j|< \varrho_j -1$ for all $j\in \N$ and
    \begin{align*}
        \sum_{j=1}^\infty |z_j| j^{-s} &\leq \sum_{j=1}^\infty |c_j|j^{-s} + \sum_{j=1}^\infty |z_j-c_j|j^{-s}\\
        &\leq \sum_{j=1}^\infty j^{-s} +\sum_{j=1}^\infty (\varrho_j -1)j^{-s} \\
        &<\infty.
    \end{align*}
This shows that $S_{s,r}\subset \ell^1_{-s}(\N)_{\C}$. Since the Bernstein-ellipses are closed, $z_j\in \mathcal E_{\varrho_j}$ if and only if $\varrho^*(z_j) \leq \varrho_j$. Therefore, 
\begin{equation*}
    R^*(z) \leq \sum_{j=1}^\infty (\varrho_j-1)j^{-s}<r.
\end{equation*}
Conversely, let $z\in \ell^1_{-s}(\N)_{\C}$ and suppose that $R^*(z)< r$. We can choose a sequence $(\varepsilon_j)_{j\in \N} \in (0,\infty)^{\N}$, such that 
\begin{equation*}
    \sum_{j=1}^\infty \varepsilon_j j^{-s} < r - R^*(z)
\end{equation*}
and set $\varrho_j \coloneqq \varrho_*(z_j) + \varepsilon_j$.
Then $z_j \in \mathcal E_{\varrho_j}$ for all $j\in \N$, and 
\begin{equation*}
    \sum_{j=1}^\infty (\varrho_j - 1)j^{-s} = R^*(z) + \sum_{j=1}^\infty \varepsilon_j j^{-s}<r
\end{equation*}
which implies that $z\in S_{s,r}$.
    
\paragraph{$S_{s,r}$ is open and bounded in $\ell^1_{-s}(\N)_{\C}$}
Since $R^*$ is continuous and $S_{s,r} = R^{-1}((0,r))$, it must be open as the pre-image of an open set. If $z\in S_{s,r}$, there exists an admissible sequence $\boldsymbol{\varrho}$ such that $z\in \mathcal E_{\boldsymbol{\varrho}}^\infty$ and hence
\begin{equation*}
    \|z\|_{\ell^1_{-s}} \leq \sum_{j=1}^\infty \rho_j j^{-s}
     \sum_{j=1}^\infty j^{-s} + r
\end{equation*}
which is finite, proving that $S_{s,r}$ is bounded.

\begin{lemma}
    Let $Q$, $\mathcal D_{\C}$ and $\mathcal O_{\C}$ as in Definition \ref{def: holomorphic operator}, such that 
    \begin{equation*}
        \sigma_s(Q)\subset \mathcal D_{\C}\subset \mathcal O_{\C}.
    \end{equation*}
    Then there exist $0<r<r'$, such that $Q \subset S_{s,r}\subset \sigma_{s}^{-1}(\mathcal D_{\C})\subset S_{s,r'}$.
\end{lemma}
\begin{proof}
    By Lemma~\ref{lemma: weighted seq space}, $\sigma_s: \ell^1_{-s}(\N)_{\C}\to \mathcal X_{\ell^1, \C}$ is continuous, and thus $\sigma_{s}^{-1}(\mathcal D_{\C})$ is open in $\ell^1_{-s}(\N)_{\C}$. Moreover, $Q\subset \sigma_{s}^{-1}(\mathcal D_{\C})$ and $Q$ is compact in $\ell^1_{-s}(\N)_{\C}$. Therefore, 
    \begin{equation*}
        r^* \coloneqq \operatorname{dist}(Q, \partial \sigma_{s}^{-1}(\mathcal D_{\C})) = \inf\{\|x-y\|: x\in Q, y\in \partial\sigma_{s}^{-1}(\mathcal D_{\C})\}>0.
    \end{equation*}
    Choose $r\in (0, r^*)$. For $z\in S_{s,r}$, it holds that 
    \begin{equation*}
        \operatorname{dist}(z, Q) \leq R^*(z) < r
    \end{equation*}
    which proves that $z\in \sigma_{s}^{-1}(\mathcal D_{\C})$. By definition, it also holds that $Q \subset S_{s,r}$. For the last inclusion, note that $\sigma_s$ is an isometric isomorphism between $\ell^1_{-s}$ and $\mathcal X_{\ell^1}$. Since $\mathcal D_{\C}$ is bounded in $\mathcal X_{\ell^1, \C}$, so must $\sigma_{s}^{-1}(\mathcal D_{\C})$ be in $\ell^1_{-s}(\N, \C)$, i.e. there exists a $M>0$, such that 
    \begin{equation*}
        \|z\|_{\ell^1_{-s}}\leq M, \qquad \text{for all } z\in \sigma_{s}^{-1}(\mathcal D_{\C}).
    \end{equation*}
    By \eqref{eq: weighted bound}, it holds that 
    \begin{equation*}
        R^*(z)\leq 2 \|z\|_{\ell^1_{-s}} \leq 2M.
    \end{equation*}
    Choosing $r' = 2M +1$ yields that $R^*(z)< r'$ and implies that $\sigma_{s}^{-1}(\mathcal D_{\C}) \subset S_{s,r'}$, which proves the claim.
\end{proof}

\section{Proofs from Section \ref{section: metric entropy}}\label{appendix: proof section 3}
\subsection{Proof of Proposition \ref{proposition: L2 metric entropy UB}}\label{appendix: proof of L2 metric entropy UB}
The following lemma is used to optimize the upper bound for the $L^2$ metric entropy of $\mathcal H^{s,t}(L;\C)$ in Proposition \ref{proposition: L2 metric entropy UB}.
\begin{lemma}\label{lemma: constr opt prob L2}
    Let $\alpha>1$, $\beta>0$, $N\in \N$ and $\varepsilon>0$. Consider the constrained optimization problem 
    \begin{align}\label{eq: constr opt L2}
        \min_{\varepsilon_1, \dots, \varepsilon_N>0} \sum_{i=1}^N \left(\frac{i^{-\alpha}}{\varepsilon_i}\right)^{1/\beta}, \quad \text{subject to } \sum_{i=1}^N \varepsilon_i^2 \leq \frac{\varepsilon^2}{2}\; 
    \end{align}
    The unique minimizer is
    \begin{align}\label{eq: eps_i sol}
        \varepsilon_i = \frac{\varepsilon i^{-\frac{\alpha}{1+2\beta}}}{\sqrt{2}(\sum_{i=1}^N i^{-\frac{2\alpha}{1+2\beta}})^{1/2}}, \qquad i = 1, \dots, N.
    \end{align}
    at which the objective equals 
    \begin{align}\label{eq: constr opt minimum L2}
        \left(\frac{\sqrt{2}}{\varepsilon}\right)^{1/\beta}\left(\sum_{i=1}^N i^{-\frac{2\alpha}{1+2\beta}}\right)^{\frac{1+2\beta}{2\beta}}.
    \end{align}
    Particularly, if $N \asymp \varepsilon^{-1/(\alpha-1/2)}$ and $\kappa\coloneqq\min\{\alpha-\frac{1}{2},\beta\}$, then for every $\tau>0$ and all sufficiently small $\varepsilon>0$, 
    \begin{align}\label{eq: metric entropy upper bound L2}
        \sum_{i=1}^N \left(\frac{i^{-\alpha}}{\varepsilon_i}\right)^{1/\beta} \lesssim \varepsilon^{-\frac{1}{\kappa} - \tau}.
    \end{align}
\end{lemma}
\begin{proof}
    The objective is strictly convex for $\varepsilon_1, \dots, \varepsilon_N>0$ and strictly decreasing in each $\varepsilon_i$. Hence, at the minimizer, $\sum_{i=1}^N \varepsilon_i^2 = \varepsilon^2/2$ and the unique minimizer is characterized by the Lagrange multipliers. Consider
    \begin{equation*}
        f(\varepsilon_1, \dots, \varepsilon_N, \lambda) = \sum_{i=1}^N \left(\frac{i^{-\alpha}}{\varepsilon_i}\right)^{1/\beta} + \lambda\left(\sum_{i=1}^N \varepsilon_i^2 - \frac{\varepsilon^2}{2}\right).
    \end{equation*}
    For each $i = 1, \dots, N$, solving $\frac{\partial f}{\partial \varepsilon_i}=0$ is equivalent to  $-\frac{1}{\beta}i^{-\frac{\alpha}{\beta}} \varepsilon_i^{-\frac{1}{\beta}-1} + 2\lambda \varepsilon_i=0$.
    Solving for $\varepsilon_i$ yields the choice
    \begin{align}\label{eq: eps_i}
        \varepsilon_i = C_\beta i^{-\frac{\alpha}{1+2\beta}}, \qquad C_\beta\coloneqq \left(\frac{1}{2\lambda \beta}\right)^{\frac{\beta}{1+2\beta}}.
    \end{align}
    Now set $S_N \coloneqq \sum_{i=1}^N i^{-\frac{2\alpha}{1+2\beta}}$. Substituting \eqref{eq: eps_i} in the constraint $\sum_{i=1}^N \varepsilon_i^2 = \varepsilon^2/2$, yields $C_\beta^2 S_N = \varepsilon^2/2$ and thus $ C_\beta = \varepsilon/\sqrt{2S_N}$.
    Substituting $C_\beta$ back into \eqref{eq: eps_i} proves \eqref{eq: eps_i sol}. 
    The objective at the minimizer is thus given by 
    \begin{align*}
        f(\varepsilon_1, \dots, \varepsilon_N, \lambda)
        &=
        C_\beta^{-1/\beta}
        \sum_{i=1}^N
        i^{-\frac{2\alpha}{1+2\beta}}\\
        &=
        \left(\frac{\sqrt{2}}{\varepsilon}\right)^{1/\beta}
        S_N^{\frac{1+2\beta}{2\beta}},
    \end{align*}
    which proves \eqref{eq: constr opt minimum L2}.
    It remains to estimate $S_N$ in terms of $\varepsilon$. By integral comparison, 
    \begin{align*}
        S_N \lesssim \begin{cases}
            N^{1 - \frac{2\alpha}{1+2\beta}}     &2\alpha <1+2\beta\\
            \log(N)                     &2\alpha = 1+2\beta\\
            1                           &2\alpha > 1+ 2\beta
        \end{cases}
    \end{align*}
    Assume now that $N \sim \varepsilon^{-1/(\alpha-1/2)}$. If $\alpha - 1/2 > \beta$, then $S_N \lesssim 1$ and we get the upper bound $\varepsilon^{-1/\beta}$. If $\alpha - 1/2 < \beta$, then we get the upper bound 
    \begin{equation*}
        \varepsilon^{-\frac{1}{\beta}}N^{\frac{1+ 2\beta - 2\alpha}{2\beta}} \asymp \varepsilon^{-\frac{1}{\alpha-1/2}}.
    \end{equation*}
    Lastly, if $\alpha -1/2 = \beta$, then we get the upper bound 
    \begin{equation*}
        \varepsilon^{-\frac{1}{\beta}}(\log(N))^{\frac{1+ 2\beta - 2\alpha}{2\beta}} \asymp \varepsilon^{-\frac{1}{\alpha-1/2}}.
    \end{equation*}
    Since $N$ is polynomial in $\varepsilon^{-1}$, for every $\tau>0$, we can bound the logarithmic factor by $\varepsilon^{-\tau}$ for all sufficiently small $\varepsilon>0$. Combining the three cases yields \eqref{eq: metric entropy upper bound L2}.
\end{proof}

\subsection{Proof of Proposition \ref{proposition: Linfty metric entropy UB}}\label{appendix: proof of Linfty metric entropy UB}
Similar to Appendix \ref{appendix: proof of L2 metric entropy UB}, the following Lemma is used to optimize the upper bound for the $L^\infty$-metric entropy upper bounds for $\mathcal H^{s,t}(L;\mathcal Y_{\C})$.
\begin{lemma}\label{lemma: constr opt prob Linfty}
    Let $\alpha>1, \beta>0$, $N\in \N$, and $\varepsilon>0$. Consider the constrained optimization problem 
    \begin{align}\label{eq: constr opt Linfty}
        \min_{\varepsilon_1, \dots, \varepsilon_N} \sum \left(\frac{i^{-\alpha}}{\varepsilon_i}\right)^{1/\beta}, \quad \text{subject to } \sum_{i=1}^N \varepsilon_i \leq \frac{\varepsilon}{2}.
    \end{align}
    Its unique minimizer is
    \begin{align}\label{eq: eps_i sol Linfty}
        \varepsilon_i = \frac{\varepsilon i^{-\frac{\alpha}{1+\beta}}}{2(\sum_{i=1}^N i^{-\frac{\alpha}{1+\beta}})}.
    \end{align}
    At this minimizer, the objective equals 
    \begin{align}\label{eq: constr opt minimum Linfty} 
        \left(\frac{2}{\varepsilon}\right)^{1/\beta} \left( \sum_{i=1}^N i^{-\frac{\alpha}{1+\beta}} \right)^{\frac{1+\beta}{\beta}}. 
    \end{align}
    Particularly, if $N\asymp \varepsilon^{-1/(\alpha-1)}$ and set $\kappa_\infty = \min\{\alpha-1, \beta\}$. Then for every $\tau>0$, and all sufficiently small $\varepsilon>0$,
    \begin{align}\label{eq: metric entropy upper bound Linfty}
        \sum_{i=1}^N \left(\frac{i^{-\alpha}}{\varepsilon_i}\right)^{1/\beta} \lesssim \varepsilon^{-\frac{1}{\kappa_\infty} - \tau} 
    \end{align}
\end{lemma}
\begin{proof}
    The proof is analogous to the proof of Lemma \ref{lemma: constr opt prob L2}, with the same objective function but a different constraint $\sum_{i=1}^N \varepsilon_i \leq \varepsilon/2$. As in the proof of Lemma \ref{lemma: constr opt prob L2}, the objective is strictly convex and strictly decreasing in each $\varepsilon_i>0$. Hence the constraint is an equality at the unique minimizer. The corresponding Lagrangian is
    \begin{equation*}
        f(\varepsilon_1, \dots, \varepsilon_N, \delta) = \sum_{i=1}^N \left(\frac{i^{-\alpha}}{\varepsilon_i}\right)^{1/\beta} + \lambda\left(\sum_{i=1}^N \varepsilon_i - \frac{\varepsilon}{2}\right),
    \end{equation*}
    and solving for the stationary points is equivalent to $-\frac{1}{\beta}i^{-\frac{\alpha}{\beta}} \varepsilon_i^{-\frac{1}{\beta}-1} + \lambda =0$.
    The solution is
    \begin{align}\label{eq: eps_i Linfty}
        \varepsilon_i = C_\beta i^{-\frac{\alpha}{1+\beta}}, \qquad C_\beta\coloneqq \left(\frac{1}{\lambda \beta}\right)^{\frac{\beta}{1+\beta}}
    \end{align}
    Now set $S_N \coloneqq \sum_{i=1}^N i ^{-\frac{\alpha}{1+ \beta}}$.
    Substituting \eqref{eq: eps_i Linfty} into the constraint $\sum_{i=1}^N \varepsilon_i = \varepsilon/2$ yields $C_\beta = \frac{\varepsilon}{2S_N}$.
    Substituting this back into \eqref{eq: eps_i Linfty} proves \eqref{eq: eps_i sol Linfty}. The corresponding objective is 
    then given by 
    \begin{equation*}
        f(\varepsilon_1, \dots, \varepsilon_N, \lambda) = \left(\frac{2}{\varepsilon}\right)^{\frac{1}{\beta}} S_N^{\frac{1+ \beta}{\beta}}
    \end{equation*}
    from which \eqref{eq: constr opt minimum Linfty} follows.
    
    Lastly, we estimate $S_N$ in terms of $\varepsilon$. As in the proof of Lemma \ref{lemma: constr opt prob L2}, an integral comparison bound yields 
    \begin{align*}
        S_N \lesssim  \begin{cases}
            N^{1 - \frac{\alpha}{1+\beta}}     &\alpha -1<\beta\\
            \log(N)                     &\alpha = 1+\beta\\
            1                           &\alpha -1 >  \beta
        \end{cases}
    \end{align*}
    Assuming $N \sim \varepsilon^{-1/(\alpha-1)}$ and arguing similarly as in the corresponding cases of Lemma \ref{lemma: constr opt prob L2}, we obtain
    \begin{align*}
        \left( \frac{2}{\varepsilon}\right)^{1/\beta} S_N^{\frac{1+ \beta}{\beta}} \lesssim \begin{cases}
            \varepsilon^{-\frac{1}{\alpha -1}} & \alpha - 1 < \beta \\
            \varepsilon^{-\frac{1}{\beta} - \tau} & \alpha - 1 = \beta\\
            \varepsilon^{-\frac{1}{\beta}} & \alpha - 1 > \beta
        \end{cases}
    \end{align*}
    for any $\tau>0$ and sufficiently small $\varepsilon>0$. The claimed bound \eqref{eq: metric entropy upper bound Linfty} follows.
\end{proof}

\begin{proof}[Proof of Proposition \ref{proposition: Linfty metric entropy UB}]
    We first consider the class of holomorphic functionals $\mathcal H^s(L; \C)$. Fix $\tau>0$, let $\delta>0$ be sufficiently small such that $s' = s- \delta >1$ and such that $1/(s'-1) \leq 1/(s-1) + \tau$. Let $g\in \mathcal H^{s}(L; \C)$ and $g^\infty_\nu$ be the corresponding Legendre coefficients in the expansion w.r.t. $(P_\nu)_{\nu \in \mathcal F}$.  As in the proof of Proposition \ref{proposition: L2 metric entropy UB}, by Theorem \ref{theorem: coefficient bounds} there exists an enumeration $( \nu_{i})_{i\in \N}$ and a constant $C>0$, both independent of $g$, such that 
    \begin{equation*}
        |g^\infty_{\nu_i}| \leq C i^{-s'}.
    \end{equation*}
    By Theorem \ref{theorem: realvalued holomorphic approximation}, 
    \begin{equation*}
        \left\| g-\sum_{i=1}^N g^\infty_{\nu_i} P_{\nu_i}\circ\sigma_s^{-1} \right\|_{\infty} \lesssim N^{-(s'-1)}.
    \end{equation*}
    We can therefore choose $N \asymp \varepsilon^{-1/(s'-1)}$, such that the error is at most $\varepsilon/2$. It remains to discretize the coefficients $g^\infty_{\nu_i}$ for $i = 1, \dots, N$. Covering $[-Ci^{-s'}, Ci^{-s'}]$ by intervals of length $\varepsilon_i$ again yields the covering numbers
    \begin{equation*}
        N_i \coloneqq N(\varepsilon_i, [-Ci^{-s'}, Ci^{-s'}], |\cdot|) \lesssim \left(\frac{i^{-s'}}{\varepsilon_i}\right).
    \end{equation*}
    Let $\mathcal C_i = \{g^\infty_{\nu_i, j_i}: j_i = 1, \dots, N_i\}$ denote the minimal $\varepsilon_i$-covering of $[-Ci^{-s'}, Ci^{-s'}]$. Assume that $\sum_{i=1}^N \varepsilon_i \leq \varepsilon$.    
    Then for every $i = 1, \dots, N$, there exists a $k_i \in \{1, \dots, N_i\}$, such that $|g^\infty_{\nu_i} - g^\infty_{\nu_i, k_i}| \leq \varepsilon_i$ and
    \begin{equation*}
        \left\|g - \sum_{i=1}^Ng^\infty_{\nu_i, k_i} P_{\nu_i} \right\|_\infty \leq \frac{\varepsilon}{2} + \sum_{i=1}^N |g^\infty_{\nu_i} - g^\infty_{\nu_i, k_i}| \leq \varepsilon.
    \end{equation*}
    
    A short calculation shows that $\varepsilon_i = \varepsilon/2N$ is an optimal choice to optimize the metric entropy upper bound 
    \begin{align*}
        \log N\!\left( \varepsilon,\mathcal H^s(L;\C),\|\cdot\|_\infty \right) &\lesssim \sum_{i=1}^N \log\!\left( 1+C\frac{i^{-s'}N}{\varepsilon} \right) \\ 
        &\lesssim N \\
        &\lesssim \varepsilon^{-1/(s'-1)} \\
        &\leq \varepsilon^{-1/(s-1)-\tau},
    \end{align*}
    where we followed almost the same calculations as in \eqref{eq: sterling}.\\
    
    Let us now consider $\mathcal H^{s,t}(L; \mathcal Y_{\C})$. Let $\tau>0$, $\delta>0$ and $s' = s- \delta$ and $N$ be chosen as in the first part of the proof. Additionally, let $t' = t- \delta$. Fix $G\in \mathcal H^{s,t}(L; \mathcal Y_{\C})$. Again, by Theorem \ref{theorem: coefficient bounds}, there exists an enumeration $(\nu_i)_{i\in \N}$ and a constant $C>0$, such that the $\mathcal Y$-valued Legendre coefficients $g^\infty_{\nu_{i}}$ satisfy
    \begin{equation*}
        \|g^\infty_{\nu_i}\|_{\mathcal Y^{t'}} \leq Ci^{-s'}.
    \end{equation*}
    As in the proof of Proposition \ref{proposition: L2 metric entropy UB}, we know that $B_{s',t', i} = \{y\in \mathcal Y: \|y\|_{\mathcal Y^{t'}} \leq Ci^{-s'}\}$ has the $\varepsilon_i$-metric entropy 
    \begin{equation*}
        \log(N_i) \coloneqq \log N(\varepsilon_i, B_{s',t',i}, \|\cdot\|_{\mathcal Y}) \asymp \left(\frac{i^{-s'}}{\varepsilon_i}\right)^{1/t'}
    \end{equation*}
    As in the scalar-valued case, we assume that $\sum_{i=1}^N \varepsilon_i \leq \varepsilon/2$. If $\mathcal C_i = \{g^\infty_{\nu_i, j_i}: j_i = 1, \dots, N_i\}$ denotes the minimal $\varepsilon_i$-covering of $B_{s',t', i}$, then for all $i = 1, \dots, N$, there exists a $k_i \in \{1, \dots, N_i\}$, such that $\|g^\infty_{\nu_i} - g^\infty_{\nu_i, k_i}\|_{\mathcal Y} \leq \varepsilon_i$ and thus
    \begin{equation*}
        \left\|G-\sum_{i=1}^Ng_{\nu_i, k_i} P_{\nu_i}\circ\sigma_s^{-1}\right\|_{\infty, \mathcal Y} = \sup_{x\in \operatorname{supp}(\gamma)}\left\|G(x) - \sum_{i=1}^Ng_{\nu_i} P_{\nu_i}(\sigma_s^{-1}(x))\right\|_{\mathcal Y}
        \leq \frac{\varepsilon}{2} + \sum_{i=1}^N\|g^\infty_{\nu_i} - g^\infty_{\nu_i, k_i}\|_{\mathcal Y} \leq \varepsilon,
    \end{equation*}
    where we used that $\|P_\nu\|_\infty \leq 1$ and the triangle inequality.
    
    The constrained optimization problem
    \begin{equation*}
        \min_{\varepsilon_1, \dots, \varepsilon_N} \sum_{i=1}^N \log(N_i), \quad \text{subject to } \sum_{i=1}^N \varepsilon_i \leq \frac{\varepsilon}{2},
    \end{equation*}
    is solved in Lemma \ref{lemma: constr opt prob Linfty} with $\alpha = s'$ and $\beta = t'$ for 
    \begin{equation*}
        \varepsilon_i = \frac{\varepsilon i^{-\frac{s'}{1+t'}}}{2(\sum_{j=1}^N j^{-\frac{s'}{1+t'}})}.
    \end{equation*}
    with the resulting $L^\infty$ metric entropy upper bound 
    \begin{equation*}
        \log N(\varepsilon, \mathcal H^{s,t}(L; \mathcal Y_{\C}), \|\cdot\|_{\infty, \mathcal Y}) \leq \varepsilon^{-\frac{1}{\min \{s'-1, t'\}}}\leq  
            \varepsilon^{-\frac{1}{\kappa_{\infty}} -\tau},
    \end{equation*}
    which finishes the proof.
\end{proof}
\subsection{Proof of Proposition \ref{proposition: Linfty metric entropy LB}}\label{appendix: proof of Linfty metric entropy LB}

We use the same packings as in the proof of Proposition \ref{proposition: L2 metric entropy LB} with $|\Omega|\geq 3^{N/4}$ and 
\begin{equation*}
    \sum_{i=1}^N (\omega_i - \omega_i')^2 \geq \frac{N}{8}, 
\end{equation*}
for all $\omega, \omega'\in \Omega$ with $\omega \neq \omega'$. 
For the class of holomorphic functionals, recall the functionals $g_\omega : \mathcal X_{\C}\to \C$, defined by 
\begin{equation*}
    g_\omega(z)=  \delta\sum_{i=1}^N \omega_i\langle z,\widetilde\psi_i\rangle_{\mathcal X},
\end{equation*}
and that for $c\in Q$
\begin{equation*}
    g_\omega(\sigma_s(c)) = \delta\sum_{i=1}^N \omega_i i^{-s}c_i.
\end{equation*}

In the proof of Proposition \ref{proposition: L2 metric entropy LB}, we have already verified that $g_\omega \in \mathcal H^{s}_{\rm lin}(L;\C)$ for sufficiently small $\delta>0$. Let $\omega, \omega' \in \Omega$ with $\omega \neq \omega'$. Since $\operatorname{supp}(\gamma) = \sigma_s(Q)$, 
\begin{align*}
    \|g_\omega - g_{\omega'}\|_\infty 
    &=  \sup_{c\in Q}|g_\omega(\sigma_s (c))- g_{\omega'} (\sigma_s(c))| \\
    &= \delta\sup_{c\in Q} \left| \sum_{i=1}^N (\omega_i-\omega_i')i^{-s}c_i \right|\\
    & = \delta \sum_{i=1}^N |\omega_i-\omega_i'|i^{-s} \\
    &\geq \delta N^{-s} \sum_{i=1}^N |\omega_i-\omega_i'|\\
    &\geq \frac{\delta}{8} N^{-(s-1)}.
\end{align*}
Here, we used that for $(a_1, \dots, a_N) \in \R^N$ and $c = (c_1, \dots, c_N)\in [-1,1]^N$, $|\sum_{i=1}^N a_i c_i|$ is maximised exactly when $c_i = \operatorname{sgn}(a_i)$. Choosing $N \asymp \varepsilon^{-1/(s-1)}$ such that last display is equal to $\varepsilon$, proves \eqref{eq: Linfty metric entropy LB functional}. 

Let us now consider the class of holomorphic operators $\mathcal H^{s,t}_{\rm lin}(L; \mathcal{Y}_{\C})$ and set $\kappa_\infty = \min\{s-1,t\}$. For $s-1 \leq t$, $\kappa_\infty = s-1$, and we again define $G_\omega : \mathcal X \to \mathcal Y$ with
\begin{equation*}
    G_\omega (z) = g_\omega (z)\phi_1.
\end{equation*}
As in Proposition \ref{proposition: L2 metric entropy LB}
, $G_\omega \in \mathcal H^{s,t}_{\rm lin}(L; \mathcal{Y}_{\C})$ for sufficiently small $\delta>0$. Additionally, 
\begin{equation*}
    \|G_\omega-G_{\omega'}\|_{\infty,\mathcal Y} = \|\phi_1\|_{\mathcal Y} \|g_\omega-g_{\omega'}\|_\infty \geq \frac{d_\Phi}{8} N^{-(s-1)}.
\end{equation*}
Choosing $N\asymp \varepsilon^{-1/(s-1)}$ such that the last term is equal to $\varepsilon$ yields the lower bound \eqref{eq: Linfty metric entropy lower bound operator} in the first case.
If $s-1>t$, we instead use the constant operators $G_\omega: \mathcal X_{\C}\to \mathcal Y_{\C}$ with
\begin{equation*}
    G_\omega(z) = \delta N^{-t-1/2} \sum_{j=1}^N\omega_j\phi_j,
\end{equation*}
for which we know that $G_\omega \in \mathcal H^{s,t}_{\rm lin}(L; \mathcal{Y}_{\C})$. Let $\omega, \omega' \in \Omega$ and $\omega \neq \omega'$. Then, since they are constant in their input variable, 
\begin{align*}
    \|G_\omega - G_{\omega'}\|^2_{\infty, \mathcal Y} 
    &= \delta^2 N^{-2t-1}\left\|\sum_{j=1}^N(\omega_j - \omega'_j)\phi_j\right\|_{\mathcal Y}\\
    &\geq \delta^2 d_\Phi^2 N^{-2t-1}\sum_{j=1}^{N}(\omega_j - \omega_j
')^2\\
&\geq \frac{\delta^2d_\Phi^2}{8} N^{-2t}.
\end{align*}
Choosing $N \asymp \varepsilon^{-1/t}$ such that the last display is equal to $\varepsilon^2$, proves the metric entropy lower bound \eqref{eq: Linfty metric entropy lower bound operator} in the second case.

\section{Proofs from Section \ref{section: lower bounds}}\label{appendix: proof section 4}
\subsection{Proof of Theorem \ref{theorem: lower bounds random design}}\label{appendix: proof of lower bounds random design}
It suffices to prove that the $KL$-divergence between two observation laws $\P_{G}$, $\P_{ G'}$ is given by 
\begin{equation*}
    KL(\P_G, \P_{G'}) = \frac{n}{2\sigma^2}\|G-G'\|_{L^2(\gamma, \mathcal Y)}^2.
\end{equation*}
Note that the law of $Y_i = G_{}(x_i) + \varepsilon$, which is denoted by $\P_{G}$, can be conditioned on the law of the design points $x_i$, in the sense that $dP_{G}(x,Y) = dP_{G}(Y|x) d\gamma(x)$. By the chain rule for the KL-divergence \parencite[Theorem 2.5.3]{cover1991elements} 
\begin{align*}
    KL(P_{G}, P_{G'}) &= \int_{\mathcal X \times \mathcal Y} \log \left(\frac{dP_{G}(x,Y)}{dP_{G'}(x,Y)}\right) dP_{G}(x,Y) \\
    &= \int_{\mathcal X}\int_{\mathcal Y} \log \left(\frac{dP_{G}(Y|x) d\gamma(x)}{dP_{G'}(Y|x) d\gamma(x)}\right) dP_{G}(Y|x) d\gamma(x) \\
    &= \int_{\mathcal X } \left (\int_{\mathcal Y}\log \left(\frac{dP_{G}(Y|x)}{dP_{G'}(Y|x)}\right) dP_{G}(Y|x) \right)d\gamma(x) \\
    & = \frac{n}{2\sigma^2} \int_{\mathcal X} \lVert G(x) - G'(x) \rVert_{\mathcal Y}^2 \;d\gamma(x) \\
    & = \frac{n}{2\sigma^2} \lVert G - G'\rVert_{L^2(\gamma, \mathcal Y)}^2,
\end{align*}
where we used that the conditional law $\P_{G'}(\cdot|x)$ is a Gaussian, with isonormal covariance structure on $\mathcal Y$.
The rest of the proof is the same as in Theorem \ref{theorem: Lower Bounds} and Theorem \ref{theorem: lower bounds functional}.

\printbibliography
\end{document}